\documentclass[smallextended]{svjour3}       
\smartqed  
\usepackage{graphicx}
\usepackage{mathptmx}      
\usepackage{hyperref}

\usepackage{amsmath}
\usepackage{amssymb}
\usepackage{amsfonts}
\usepackage{color}
\usepackage{datetime}
\journalname{Mathematics of Control, Signals, and Systems}
\DeclareMathAlphabet{\bit}{OML}{cmm}{b}{it}

\def\<{\leqslant}           
\def\>{\geqslant}           

\def\d{\partial}
\def\wh{\widehat}

\def\Re{\mathrm{Re}}   

\def\cH{\mathcal{H}}   

\def\mN{\mathbb{N}}    
\def\mR{\mathbb{R}}    
\def\mC{\mathbb{C}}    

\def\Tr{\mathrm{Tr}}       
\def\rT{\mathrm{T}}        
\def\rF{\mathrm{F}}

\def\bE{\mathbf{E}}    

\def\rprod{\mathop{\overrightarrow{\prod}}}
\def\lprod{\mathop{\overleftarrow{\prod}}}

\def\esssup{\mathop{\mathrm{ess\, sup}}}

\def\bra{\langle}
\def\ket{\rangle}

\def\re{\mathrm{e}}        
\def\rd{\mathrm{d}}        

\def\fL{\mathfrak{L}}

\def\cL{\mathcal{L}}

\def\br{\mathbf{r}}
\def\x{\times}

\def\fG{\mathfrak{G}}

\def\cC{\mathcal{C}}

\def\cG{\mathcal{G}}
\def\cI{\mathcal{I}}
\def\cP{\mathcal{P}}

\def\cA{\mathcal{A}}
\def\cB{\mathcal{B}}
\def\cE{\mathcal{E}}

\def\var{\mathbf{var}}

\def\cT{{\mathcal T}}

\def\mH{\mathbb{H}}
\def\mS{\mathbb{S}}

\def\eps{\epsilon}
\def\Ups{\Upsilon}
\def\ups{\upsilon}
\def\phi{\varphi}

\def\diag{\mathop{\mathrm{diag}}}    

\begin{document}
\title{Covariance-analytic performance criteria, Hardy-Schatten norms and
Wick-like ordering of cascaded systems
\thanks{This work is supported by the Australian Research Council  grant DP210101938.}
}

\titlerunning{Covariance-analytic criteria and Hardy-Schatten norms}        

\author{Igor G. Vladimirov, \quad Ian R. Petersen}

\authorrunning{I.G.Vladimirov, I.R. Petersen}

\institute{Australian National University, Canberra, Australia\\
              \email{igor.g.vladimirov@gmail.com, i.r.petersen@gmail.com}
}

\date{Received: date / Accepted: date}

\maketitle

\begin{abstract}
This paper is concerned with linear stochastic systems whose output is a stationary Gaussian random process related by an integral operator to a standard Wiener process at the input.
We consider a performance criterion which involves the trace of an analytic function of the spectral density of the output process.  This class of ``covariance-analytic''  cost functionals includes the usual mean square and risk-sensitive criteria as particular cases. Due to the presence of the ``cost-shaping'' analytic function, the performance criterion is related to higher-order Hardy-Schatten norms of the system transfer function.  These norms have links with the asymptotic properties of cumulants of finite-horizon quadratic functionals of the system output and satisfy variational inequalities pertaining to system robustness to statistically uncertain inputs.
In the case of strictly proper finite-dimensional systems, governed in state space by linear stochastic differential equations, we develop a method for recursively computing the Hardy-Schatten norms through a recently proposed technique of rearranging cascaded linear systems, which resembles the Wick ordering of annihilation and creation operators in quantum mechanics. The resulting computational procedure involves a recurrence sequence of solutions to algebraic Lyapunov equations and represents the covariance-analytic cost as the squared  $\cH_2$-norm of an auxiliary cascaded system. These results are also compared with an alternative approach which uses higher-order derivatives of stabilising solutions of  parameter-dependent algebraic Riccati equations. 
\end{abstract}

\section{Introduction}

Multivariable linear continuous-time invariant stochastic systems, which can be represented as Toeplitz integral operators acting on an input disturbance in the form of a multidimensional Wiener process or more complicated random processes,  provide an important class of tractable models of stochastic dynamics. These models result from linearised description of physical systems such as RLC circuits with thermal noise, vehicle suspension subject to rough  terrain,  or  flexible structures interacting with turbulent flows.

The statistical properties (including the spectral density or pertaining to higher order moments) of the output of a linear system are related to those of the input. In the frequency domain, the input-output characteristics of the system are captured in its transfer function, through which the input and output spectral densities are related to each other.  In the case of a standard Wiener process at the input, the output spectral density, as a function of the frequency,  quantifies the contribution from different modes to the variance of the output process,  which coincides with the squared $\cH_2$-norm of the transfer function.  This mean square functional describes the infinite-horizon growth rate for the integral of the squared Euclidean norm of the system output over a bounded time interval. This integral yields a random variable, which depends on the output in a quadratic fashion and can be interpreted as the output energy of the system. In turn, the asymptotic behaviour of fluctuations in the output energy, quantified in terms of its variance,  is closely related to the $\cH_4$-norm of the transfer function  in an appropriate Hardy space.
The $\cH_4$-norm uses the Schatten 4-norm of matrices \cite{HJ_2007} (involving the fourth power of the singular values),  and its discrete-time version   appeared in  \cite[Lemma 2]{VKS_1996} in the context of anisotropy-based robust control.

The $\cH_2$ and $\cH_4$-norms underlie the linear quadro-quartic Gaussian (LQQG) approach \cite{VP_2010b}, which extends   LQG control towards minimising not only the first, but also the second-order cumulants of the output energy in the framework of the disturbance attenuation paradigm.    These norms are the first two elements in a sequence of the Hardy-Schatten $\cH_{2k}$-norms, $k = 1, 2, 3, \ldots$, in the corresponding spaces  \cite{S_2005} of matrix-valued transfer functions, holomorphic in the right half-plane and endowed with the $L^{2k}$-norm of their singular values (or, equivalently, the $L^k$ norm for the eigenvalues of the output spectral density).

The Hardy-Schatten norms of all orders participate (though implicitly) in the performance criteria of the risk-sensitive  and minimum entropy control theories \cite{MG_1990}. These theories use the cumulant-generating function of the output energy scaled by a risk sensitivity parameter. The resulting infinite-horizon risk-sensitive performance index is organised as a power series with respect to this parameter, with the coefficients of the series being the growth rates of the output energy cumulants.  The  cumulant rates are related to the corresponding powers of the Hardy-Schatten norms of the transfer function, which makes the risk-sensitive criterion a specific linear combination of appropriately powered $\cH_{2k}$-norms,  with the coefficients being specified by the risk sensitivity parameter.

As suggested in \cite{VP_2010b}, ``reassembling'' the terms of the risk-sensitive cost with different weights (not necessarily governed by a single parameter) gives rise to a wide class of performance criteria (in the form of linear combinations of powers of the Hardy-Schatten norms) and the corresponding  output energy cumulant (OEC) control problems for linear stochastic systems. This class contains the LQG, LQQG and risk-sensitive approaches as special cases which explore this freedom to a certain extent.
At the same time, the risk-sensitive criteria  play a distinctive role in the robust performance characteristics of the system in the form of the worst-case LQG costs  in the presence of statistical uncertainty about the input noise with a  Kullback-Leibler relative entropy description \cite{DJP_2000,PJD_2000}.
It is also relevant to equip the extended OEC approach with its own variational principle for a rational weighting of the Hardy-Schatten norms in regard to appropriately modified robustness properties of the system.

The practical implementation of the OEC approach requires, as its ingredient, the development of state-space methods for computing the Hardy-Schatten norms and related cost functionals, which is the main purpose of the present paper. We are concerned with linear time-invariant stochastic systems,   which act as an integral operator on a standard Wiener process at the input and produce a stationary Gaussian random process at the output. Starting from frequency-domain formulations, we propose a performance criterion which involves the trace of an analytic function evaluated at the output spectral density. This class
of ``covariance-analytic''  cost functionals includes the standard mean square and risk-sensitive criteria as particular cases. Due to the presence of the ``cost-shaping'' analytic function in its definition, this performance criterion admits a series expansion involving all the $\cH_{2k}$-norms of the system. In combination with the Legendre transformation, applied to trace functions of matrices \cite{L_1973,VP_2010a},  the structure of the covariance-analytic cost gives rise to a variational inequality for this functional
concerning statistical uncertainty when the system input is a more complicated Gaussian process (with correlated increments unlike the Wiener process). 

For finite-dimensional systems, governed in state space by linear stochastic differential equations  (SDEs), we develop a method for recursively computing the Hardy-Schatten norms through a recently proposed approach \cite{VP_2022}  to rearranging cascaded linear systems which arise from powers of a rational spectral density. This ``system transposition'' technique, oriented to a special  spectral factorisation problem,  resembles the Wick ordering of noncommuting annihilation and creation operators in quantum mechanics \cite{J_1997,W_1950}. The resulting computational procedure involves a recurrence sequence of solutions to algebraic Lyapunov equations (ALEs)  and represents the covariance-analytic cost as the squared  $\cH_2$-norm of an auxiliary cascaded system. These results are compared with an alternative approach which uses another set of ALEs for computing the higher-order derivatives of stabilising solutions of parameter-dependent algebraic Riccati equations.

We also mention that a related yet different class of nonquadratic integral performance criteria, involving power series over noncommuting variables, was discussed in a quantum feedback control framework, for example,   in \cite{P_2014}. Also, the discrete-time counterparts of the Hardy-Schatten norms of all orders play an important role  \cite{VKS_1999} in  the anisotropy-constrained versions of the $\ell^2$-induced operator norms of systems.

The paper is organised as follows.
Section~\ref{sec:sys} specifies the class of linear stochastic systems and covariance-analytic costs under consideration.
Section~\ref{sec:cum} describes the Hardy-Schatten norms and discusses their links with the growth rates of the  output energy cumulants.
Section~\ref{sec:var}  provides a variational inequality for the covariance-analytic functional and related robustness properties.
Section~\ref{sec:fin} specifies a class of strictly proper finite-dimensional systems in state space.
Section~\ref{sec:inf} obtains a spectral factorisation for a particular cascade of systems.
Section~\ref{sec:HS} applies this factorisation to the state-space computation of the
Hardy-Schatten norms and covariance-analytic costs.
Section~\ref{sec:diff} discusses  an alternative approach to their computation using a parameter-dependent  Riccati equation.
Section~\ref{sec:num} illustrates both procedures for computing the Hardy-Schatten norms by a numerical example.
Section~\ref{sec:conc} provides concluding remarks.

\section{Linear stochastic systems and covariance-analytic costs}
\label{sec:sys}

Let $Z:= (Z(t))_{t\in \mR}$ be an $\mR^p$-valued   stationary Gaussian random process \cite{IR_1978} produced as the output of a linear causal time-invariant
system from an $\mR^m$-valued   standard Wiener process $W:= (W(t))_{t\in \mR}$ at  the input\footnote{the process $W$ does not require a particular initialisation in the infinitely distant past, since only its increments are relevant here}: \begin{equation}
\label{ZfW}
    Z(t)
    =
    \int_{-\infty}^t
    f(t-\tau)
    \rd W(\tau),
    \qquad
    t\in \mR.
\end{equation}
The impulse response  $f: \mR_+ \to \mR^{p\x m}$ to the incremented input (with $\mR_+:= [0,+\infty)$) is assumed to be square integrable,  $f \in L^2(\mR_+, \mR^{p\x m})$, thus securing the existence of the Ito integral in (\ref{ZfW}),  along with the isometry
\begin{equation}
\label{iso}
    \bE(|Z(t)|^2)
    =
    \int_{\mR_+}
    \|f(\tau)\|_\rF^2 \rd \tau
    <
    +\infty
\end{equation}
for the variance of the process $Z$. Here, $\bE(\cdot)$ is expectation, and $\|M\|_\rF:= \sqrt{\bra M,M\ket_\rF}$ is the Frobenius norm \cite{HJ_2007}  generated by the inner product $\bra M, N\ket_\rF:= \Tr (M^*N)$ of real or complex matrices, with $(\cdot)^* := (\overline{(\cdot)})^{\rT}$ the complex conjugate transpose. We will also write $\|M\|_Q:= \|\sqrt{Q} M\|_\rF =  \sqrt{\Tr(M^* Q M)}$ for the weighted Frobenius (semi-) norm associated  with a positive (semi-) definite Hermitian matrix $Q$. The process $Z$ in (\ref{ZfW}) has zero mean and a continuous covariance function $K \in C(\mR, \mR^{p\x p})$, which specifies the two-point covariance matrix
\begin{align}
\nonumber
    K(s-t)
    & :=
    \bE(Z(s) Z(t)^\rT)
    =
    \int_{-\infty}^{\min(s,t)}
    f(s-\tau)f(t-\tau)^{\rT}
    \rd
    \tau\\
\label{EZZ}
    & =
    \frac{1}{2\pi}
    \int_{\mR}
    \re^{i\omega (s-t)}
    S(\omega)
    \rd \omega
    =
    K(t-s)^{\rT},
    \qquad
    s, t \in \mR,
\end{align}
and its particular case, the one-point covariance matrix
\begin{equation}
\label{K0}
    K(0)
    :=
    \bE(Z(t) Z(t)^\rT)
    =
    \int_{\mR_+}
    f(\tau)f(\tau)^{\rT}
    \rd \tau=
    \frac{1}{2\pi}
    \int_{\mR}
    S(\omega)
    \rd \omega,
\end{equation}
whose trace is given by (\ref{iso}). Here, $S: \mR\to \mH_p^+$ (with $\mH_p^+$ the set of positive semi-definite matrices in the subspace $\mH_p$ of complex Hermitian  matrices  of order $p$) is the spectral density of the process $Z$:
\begin{equation}
\label{S}
    S(\omega)
    :=
    \int_{\mR}
    \re^{-i \omega t}
    K(t)
    \rd t
    =
    \wh{F}(\omega) \wh{F}(\omega)^*.
\end{equation}
It is expressed in terms of the Fourier transform of the impulse response of the system given by
\begin{equation}
\label{Fhat}
    \wh{F}(\omega)
    :=
    F(i\omega)
    =
    \int_{\mR_+}
    \re^{-i \omega t}
    f(t)
    \rd t,
    \qquad
    \omega \in \mR,
\end{equation}
which is the
boundary value of the transfer function
\begin{equation}
\label{Ff}
  F(s)
  :=
  \int_{\mR_+}
  \re^{-st} f(t)\rd t,
\end{equation}
holomorphic in the open right half-plane  $\{s \in \mC: \Re s >0\}$.
Since the impulse response $f$ is square integrable, (\ref{Ff}) belongs to the Hardy space
$\cH_2^{p\x m}$ of $\mC^{p\x m}$-valued functions of a complex
variable, holomorphic in the open right half-plane and equipped with the
$\cH_2$-norm
\begin{equation}
\label{H2}
    \|F\|_2
     :=
    \sqrt{
        \frac{1}{2\pi}
        \int_{\mR}
        \|\wh{F}(\omega)\|_\rF^2
        \rd
    \omega}
    =
    \sqrt{
        \int_{\mR_+}
        \|f(t)\|_\rF^2
        \rd t
    }
    =
    \sqrt{
        \Tr K(0)
    },
\end{equation}
where the Plancherel identity is combined with (\ref{iso})--(\ref{Fhat}), and
$\|\wh{F}(\omega)\|_\rF^2 = \Tr S(\omega)$ is related to the output spectral
density (\ref{S}).

If the transfer function $F$ pertains to a plant-controller interconnection, and the process $Z$ consists of the dynamic variables whose small values are a preferred behaviour of the closed-loop system, the corresponding performance criteria usually employ cost functionals to be minimised over the controller parameters.  Such functionals include the mean square cost $\|F\|_2^2$ from (\ref{H2}) and the quadratic-exponential functional
\begin{equation}
\label{QE}
    \Xi(\theta)
    :=
    \lim_{T\to +\infty}
    \Big(
    \frac{1}{T}
    \ln
    \bE
    \re^{\frac{\theta}{2}\cE_T}
    \Big)
    =
    -
    \frac{1}{4\pi}
    \int_{\mR}
    \ln\det
    (I_p - \theta S(\omega))
    \rd \omega ,
\end{equation}
used in the risk-sensitive  and minimum entropy
control theories \cite{MG_1990}.   Here,
for a finite
time horizon $T>0$, the random variable
\begin{equation}
\label{ET}
    \cE_T
    :=
    \int_{0}^{T}
    |Z(t)|^2
    \rd t
\end{equation}
can be interpreted as an ``output energy'' of the system over the
time interval $[0,T]$, and $\theta>0$ is a risk sensitivity parameter  constrained by
\begin{equation}
\label{tgood}
    \theta < \|F\|_\infty^{-2},
\end{equation}
with
\begin{equation}
\label{Hinf}
    \|F\|_\infty
    :=
    \esssup_{\omega \in \mR}
    \|\wh{F}(\omega)\|
    =
    \esssup_{\omega\in \mR}\sqrt{\lambda_{\max}(S(\omega))}
\end{equation}
the $\cH_\infty$-norm of the transfer function $F$. Also, $\|\cdot\|$ is the operator matrix norm, and $\lambda_{\max}(\cdot)$ is the largest eigenvalue of a matrix with a real spectrum. The corresponding ``energy rate''
\begin{equation*}
\label{eT}
    \eps_T
    :=
    \frac{1}{T}\cE_T
\end{equation*}
satisfies
\begin{equation*}
\label{EF}
    \bE \eps_T = \|F\|_2^2,
    \qquad
    T >0,
\end{equation*}
along with the mean square convergence
$$    \mathop{\rm l.i.m.}_{T\to +\infty}
    \eps_T
    =
    \| F\|_2^2.
$$
The latter holds  in view of the relations
$$
    \var(\eps_T)
    =
    \bE((\eps_T - \|F\|_2^2)^2)
    \sim
    \frac{2}{T} \|F\|_4^4,
    \qquad
    {\rm as}\ T\to +\infty,
$$
from \cite[Lemma 1]{VP_2010b},
provided the output spectral density $S$ in (\ref{S}) is square integrable (and hence, so also is the covariance function $K$ in (\ref{EZZ})), in which case the transfer function $F$ of the system has a finite $\cH_4$-norm\footnote{its discrete-time counterpart is employed as a subsidiary construct in the anisotropy-based control theory; see, for example,   \cite[Lemmas 2, 5]{VKS_1996}}
\begin{equation}
\label{H4}
    \|F\|_4
     :=
    \sqrt[4]{
    \frac{1}{2\pi}
    \int_{\mR}
    \|S(\omega)\|_\rF^2
    \rd \omega}
    =
    \sqrt[4]{
    \int_{\mR}
    \|K(t)\|_\rF^2
    \rd t
    }
    =
    \sqrt[4]{
    2
    \int_{\mR_+}
    \|K(t)\|_\rF^2
    \rd t
    }.
\end{equation}
The integrand $\|S(\omega)\|_\rF^2 = \Tr ((\wh{F}(\omega) \wh{F}(\omega)^*)^2)$ is the
fourth power of the Schatten 4-norm \cite[p.~441]{HJ_2007} of the
matrix $\wh{F}(\omega)$;  see also \cite{S_2005}.    In (\ref{H4}), use is made of the Plancherel identity, which is  applied to (\ref{S}) in combination with (\ref{EZZ}) and the invariance of the Frobenius norm under the matrix transpose and complex conjugation, whereby $\|K(t)\|_\rF = \|K(-t)\|_\rF$ for any $t \in \mR$.
The transfer functions $F$ with
$\|F\|_4<+\infty$ form a normed space $\cH_4^{p\x m}$.

The quadratic cost $\|F\|_2^2$, its ``quartic'' counterpart $\|F\|_4^4$,    and the risk-sensitive cost (\ref{QE}) are particular cases of a ``covariance-analytic'' functional
\begin{equation}
\label{phiFF}
    J_\phi(F)
    :=
    \frac{1}{2\pi}
  \int_{\mR}
  \Tr \phi(S(\omega))
  \rd \omega,
\end{equation}
specified by a function $\phi$ of a complex variable,  which is analytic in a neighbourhood of the interval  $[0, \|F\|_\infty^2]$ in the complex plane, takes real values on this interval, satisfies
\begin{equation}
\label{phi0}
  \phi(0) = 0
\end{equation}
and is evaluated \cite{H_2008} at the matrix (\ref{S}).
In the case of absolutely integrable impulse responses $f$,    the condition (\ref{phi0}) is necessary for the convergence of the integral in (\ref{phiFF}) since,  in that case,  $\lim_{\omega\to \infty}S(\omega) = 0$ in view of the Riemann–Lebesgue Lemma applied to (\ref{Fhat}), so that $\lim_{\omega\to \infty}\phi(S(\omega)) = \phi(0)I_p$.
For example, the risk-sensitive cost   (\ref{QE}) corresponds to (\ref{phiFF}) with the function
\begin{equation}
\label{QEphi}
    \phi_\theta(z)
    :=
    -\frac{1}{2}\ln(1-\theta z)
    =
    \frac{1}{2}
    \sum_{k =1}^{+\infty}
    \frac{1}{k}
    (\theta z)^k
    =
    \phi_1(\theta z),
\end{equation}
which satisfies (\ref{phi0}) and is analytic in the open disc $\{z \in \mC: |z|< \frac{1}{\theta}\}$ containing the interval $[0, \|F\|_\infty^2]$ in view of (\ref{tgood}).  This correspondence follows from the  identity $\ln\det M = \Tr \ln M$ for nonsingular matrices,  applied to the integrand in (\ref{QE}). Alternatively, $J_{\phi_\theta}(F)$  can be evaluated  by applying (\ref{phiFF}) to a rescaled transfer function as $J_{\phi_1}(\sqrt{\theta} F)$, where $\phi_1$ is given by (\ref{QEphi}) with $\theta = 1$.

For a wider class of spectral densities $S$  with a finite (and not necessarily zero) limit $S(\infty):= \lim_{\omega \to \infty} S(\omega)$, a necessary condition for convergence  of the integral in (\ref{phiFF}), similar to (\ref{phi0}), is provided by
\begin{equation}
\label{phiS}
  \Tr \phi(S(\infty)) = 0.
\end{equation}
It is also possible to incorporate a frequency-dependent weight $u\in L^1(\mR, \mR_+)$  in (\ref{phiFF}) by considering the quantity
$$
        \frac{1}{2\pi}
  \int_{\mR}
        u(\omega)
  \Tr \phi(S(\omega))
  \rd \omega,
$$
which extends such functionals to transfer functions $F \in \cH_\infty^{p\x m}$ with a bounded spectral density $S$, not necessarily vanishing at infinity, and makes the condition    (\ref{phiS}) redundant.   However, for simplicity,   we leave this extension aside and will  only  discuss the costs (\ref{phiFF}).

\section{Hardy-Schatten norms and output energy  cumulants}
\label{sec:cum}

In accordance with (\ref{QEphi}), the Taylor series expansion of the risk-sensitive cost (\ref{QE}) can be  represented as
\begin{equation}
\label{XiF}
    \Xi(\theta)
    =
    \frac{1}{4\pi}
    \sum_{k =1}^{+\infty}
    \frac{1}{k}
    \theta^k
    \int_\mR
    \Tr
    (S(\omega)^k)
    \rd \omega
    =
    \frac{1}{2}
    \sum_{k =1}^{+\infty}
    \frac{1}{k}
    \theta^k
    \|F\|_{2k}^{2k},
\end{equation}
so that
\begin{equation}
\label{FXi}
    \|F\|_{2k}^{2k}
    =
    \frac{2}{(k-1)!}\Xi^{(k)}(0),
    \qquad
    k \in \mN
\end{equation}
 (with $\mN$ the set of positive integers). Here, use is made of the Hardy-Schatten norms of the transfer function $F$ defined in terms of the output spectral density $S$ from (\ref{S}) by
\begin{equation}
\label{H2k}
    \|F\|_{2k}
     :=
    \sqrt[2k]{
    \frac{1}{2\pi}
    \int_{\mR}
    \Tr(S(\omega)^k)
    \rd \omega}
\end{equation}
on the corresponding
spaces $\cH_{2k}^{p\x m}$. The integrand in (\ref{H2k}) is the $2k$th power of the Schatten $2k$-norm \cite[p.~441]{HJ_2007}
$$
    \sqrt[2k]{\Tr(S(\omega)^k)}
    =
    \sqrt[2k]{\Tr ((\wh{F}(\omega) \wh{F}(\omega)^*)^k)}
$$
of the matrix
$\wh{F}(\omega)$. The $\cH_2$
and $\cH_4$-norms of the system in (\ref{H2}), (\ref{H4}) are particular cases of (\ref{H2k}) with  $k=1,2$.
 The representation (\ref{XiF}) can also be viewed as a series over the powers $\|F\|_{2k}^{2k}$ of the Hardy-Schatten norms,  with the coefficients $\frac{1}{2k}\theta^k$ specified by a single risk sensitivity parameter $\theta$.

Returning from the risk-sensitive example to the general case, the Taylor series expansion of the function $\phi$ in (\ref{phiFF}) allows the covariance-analytic cost to be represented as an infinite\footnote{unless $\phi$ is a polynomial} linear combination
\begin{equation}
\label{JHS}
    J_\phi(F)
    =
    \sum_{k=1}^{+\infty}
    \phi_k
    \|F\|_{2k}^{2k},
    \qquad
    \phi_k :=     \frac{1}{k!}
    \phi^{(k)}(0)
\end{equation}
of the appropriately powered Hardy-Schatten norms (\ref{H2k}), with the coefficients
$\phi_k \in \mR$  specified by $\phi$ as a ``cost-shaping'' function.

We will now discuss several properties of the $\cH_{2k}$-norms in regard to their probabilistic origin from the output energy (\ref{ET}) of the system.
In what follows, the trivial case of identically zero transfer functions $F$ (with $\|F\|_\infty= 0$) is excluded from consideration.

\begin{lemma}
\label{lem:norms}
The Hardy-Schatten norms (\ref{H2k}) are related to the $\cH_2$- and $\cH_\infty$-norms in (\ref{H2}), (\ref{Hinf}) by the inequality
\begin{equation}
\label{HHH}
  \|F\|_{2k}
  \<
  \|F\|_\infty
  \sqrt[k]{\|F\|_2/\|F\|_\infty},
  \qquad
  k \in \mN.
\end{equation}
\end{lemma}
\begin{proof}
At any frequency $\omega\in \mR$,  the output spectral density (\ref{S}) is a positive semi-definite Hermitian matrix, which  satisfies
$S(\omega) \preccurlyeq \|F\|_\infty^2 I_p$ almost everywhere in view of (\ref{Hinf}), and hence,
\begin{equation}
\label{SSS}
    \Tr (S(\omega)^{k+1})
    =
    \Tr (S(\omega)^{k/2} S(\omega)S(\omega)^{k/2})
    \<
    \|F\|_\infty^2 \Tr (S(\omega)^k)
\end{equation}
for any $k \in \mN$. By integrating both sides of (\ref{SSS}) over $\omega$ and using (\ref{H2k}), it follows that
$$
    \|F\|_{2k+2}^{2k+2}
    =
    \frac{1}{2\pi}
    \int_{\mR}
    \Tr(S(\omega)^{k+1})
    \rd \omega
    \<
    \|F\|_\infty^2
    \|F\|_{2k}^{2k},
$$
which,  by induction on $k$,  leads to
$$
    \|F\|_{2k}^{2k}
    \<
    \|F\|_\infty^{2k-2}
    \|F\|_2^2
    =
    \|F\|_\infty^{2k}
    (\|F\|_2/\|F\|_\infty)^2,
$$
thus establishing (\ref{HHH}).
\end{proof}

By (\ref{HHH}), the finiteness of the $\cH_2$ and $\cH_\infty$-norms ensures  that all the Hardy-Schatten norms (\ref{H2k})  are also finite, and hence,
\begin{equation}
\label{HHH1}
    \cH_2^{p\x m}
    \bigcap
    \cH_\infty^{p\x m}
    \subset
    \bigcap_{k\in \mN} \cH_{2k}^{p\x m}.
\end{equation}
Another corollary of (\ref{HHH}) is that
$$
    \limsup_{k\to +\infty}
    \|F\|_{2k}
    \<
    \|F\|_\infty
    \lim_{k\to +\infty}
    \sqrt[k]{\|F\|_2/\|F\|_\infty}
    =
    \|F\|_\infty.
$$
Similarly to (\ref{QE}), these norms govern the infinite-horizon asymptotic behaviour of the cumulants
\begin{equation}
\label{bCkT}
    C_{k,T}:= \d_v^k \ln \bE \re^{v\cE_T}\big|_{v=0},
    \qquad
    k \in \mN,
\end{equation}
of the output energy $\cE_T$ in (\ref{ET}) as discussed below. These cumulants are related to the corresponding moments as
\begin{equation}
\label{CPi}
     C_{k,T} = \Pi_k(\bE \cE_T, \ldots, \bE (\cE_T^k)),
     \qquad
     k \in \mN,
\end{equation}
through
$k$-variate polynomials
\begin{equation}
\label{Pi}
    \Pi_k(\mu_1, \ldots, \mu_k)
    :=
    \sum_{\ell_1, \ldots, \ell_k\> 0:\ \sum_{j=1}^k j\ell_j = k}
    \pi_{k; \ell_1, \ldots, \ell_k}
    \prod_{j=1}^k
    \mu_j^{\ell_j}
\end{equation}
(with real coefficients $\pi_{k; \ell_1, \ldots, \ell_k}$, independent of the probability distribution of $\cE_T$), which are homogeneous in the sense that
$$
    \Pi_k(\sigma \mu_1,\sigma^2 \mu_2, \ldots, \sigma^k \mu_k) = \sigma^k \Pi_k(\mu_1, \ldots, \mu_k),
    \qquad
    \sigma \in \mR,
$$
as illustrated by the first three polynomials:
$$
    \Pi_1(\mu_1)
     =
    \mu_1,
    \quad
    \Pi_2(\mu_1, \mu_2)
     =
    \mu_2 - \mu_1^2,
    \quad
    \Pi_3(\mu_1, \mu_2, \mu_3)
     =
    \mu_3 - 3\mu_1\mu_2 + 2\mu_1^3,
$$
where $\Pi_1$, $\Pi_2$ yield the mean and variance, respectively.
The following theorem, whose results were presented in a slightly different form in \cite{VP_2010b} and a quantum-mechanical version was obtained in \cite[Theorem 4]{VPJ_2018a}, is closely related to the Szeg\H{o} limit theorems for the  spectra of Toeplitz operators
\cite{GS_1958} and is given here for completeness.

\begin{theorem}
\label{th:CK}
Suppose the system (\ref{ZfW}) has finite $\cH_2$- and $\cH_\infty$-norms (\ref{H2}), (\ref{Hinf}), so that its transfer function (\ref{Ff}) satisfies
\begin{equation}
\label{Fgood1}
  F \in     \cH_2^{p\x m}
    \bigcap
    \cH_\infty^{p\x m}.
\end{equation}
Then the cumulants (\ref{bCkT}) of the output energy (\ref{ET}) have the following asymptotic growth rates
\begin{equation}
\label{bCasy}
  \lim_{T\to +\infty}
  \Big(
    \frac{1}{T}
    C_{k, T}
  \Big)
  =
  (2k-2)!!
  \|F\|_{2k}^{2k},
  \qquad
  k \in \mN,
\end{equation}
in terms of the Hardy-Schatten norms (\ref{H2k}). \hfill$\square$
\end{theorem}
\begin{proof}
In view of the inclusion (\ref{HHH1}), the fulfillment of the condition (\ref{Fgood1}) also ensures that
\begin{equation}
\label{Fgood}
    F\in \bigcap_{k\in \mN} \cH_{2k}^{p\x m}.
\end{equation}
Now, for a fixed but otherwise arbitrary $T>0$, consider a compact positive semi-definite self-adjoint  operator $K_T$ with the continuous covariance kernel (\ref{EZZ}) acting on a function $\psi\in L^2([0,T],\mC^p)$ as
\begin{equation}
\label{KT}
    K_T(\psi)(s)
    :=
    \int_{0}^{T}
    K(s-t)\psi(t)
    \rd
    t,
    \qquad
    0 \< s \< T.
\end{equation}
In view of the Fredholm determinant
formula \cite[Theorem~3.10 on p.~36]{S_2005} (see also \cite{Ginovian_1994} and references therein),
\begin{align}
\nonumber
    \ln \bE
    \re^{\frac{\theta}{2}\cE_T}
     & =
    -
    \frac{1}{2}
    \Tr
    \ln(
        \cI-\theta K_T
    )\\
\label{expquad}
     & =
     -\frac{1}{2}
     \sum_{j\in \mN}
     \ln (1-\theta \lambda_{j,T})
     =
        \frac{1}{2}
        \sum_{k\in \mN}
        \frac{1}{k}\theta^k
        \Tr(K_T^k)
\end{align}
for any
\begin{equation}
\label{tgoodT}
    \theta < 1/\br(K_T),
\end{equation}
where $\lambda_{j,T}\> 0$ are the eigenvalues  of the operator $K_T$, with $\br(K_T) =  \max_{j\in \mN}\lambda_{j,T}$ its spectral radius, and $\cI$ is the identity operator on $L^2([0,T], \mC^p)$. Note that (\ref{tgood}) implies (\ref{tgoodT}) in view of the upper bound \cite{GS_1958}
\begin{equation}
\label{KTnorm}
    \max_{j\in \mN}\lambda_{j,T}
    =
    \|K_T\| \< \|F\|_\infty^2,
    \qquad
    T >0,
\end{equation}
for the induced operator norm of $K_T$ on the Hilbert space $L^2([0,T], \mC^p)$.
The spectrum of $K_T$ satisfies
\begin{equation}
\label{Ktrace}
    \sum_{j\in \mN} \lambda_{j,T} = \Tr K_T = \int_0^T \Tr K(0)\rd t = T\Tr K(0) = T \|F\|_2^2
\end{equation}
due to the Toeplitz structure of the covariance kernel \cite{B_1988,RS_1980} and in accordance with (\ref{H2}), thus securing the convergence of the series in (\ref{expquad}).
Moreover, for any $k \in \mN$, the trace of the $k$-fold iterate
of the operator $K_T$ in (\ref{KT}) is given by a convolution-like integral
\begin{align}
\nonumber
    \Tr(K_T^k)
    & =
    \sum_{j\in \mN}\lambda_{j,T}^k\\
\nonumber
    & = \int_{[0,T]^k}
    \Tr
    (
        K(t_1-t_2)
        K(t_2-t_3)
        \x
        \ldots
        \x
        K(t_{k-1}-t_k)
        K(t_k-t_1)
    )
    \rd t_1
    \x \ldots \x
    \rd t_k,
\end{align}
whose asymptotic behaviour is described by
\begin{equation}
\label{KTasy}
  \lim_{T\to +\infty}
  \Big(
    \frac{1}{T}
    \Tr (K_T^k)
  \Big)
  =
  \frac{1}{2\pi}
  \int_\mR
  \Tr (S(\omega)^k)
  \rd \omega
  =
  \|F\|_{2k}^{2k},
\end{equation}
in view of (\ref{S}), (\ref{H2k}), (\ref{Fgood}) and \cite[Lemma 6 in Appendix C]{VPJ_2018a}. On the other hand, the left-hand side of (\ref{expquad}) is the
cumulant-generating function of $\cE_T$ evaluated at $\frac{\theta}{2}$ and, therefore, related to the cumulants (\ref{bCkT}) as
\begin{equation}
\label{expquad1}
    \ln \bE\re^{\frac{\theta}{2}\cE_T}
    =
    \sum_{k\in \mN }
    \frac{1}{k!}
    (\theta/2)^k
    C_{k,T},
\end{equation}
whereby
\begin{equation}
\label{bCK}
    C_{k,T}
    =
    (k-1)!2^{k-1}
    \Tr(K_T^k)
    =
    (2k-2)!!
    \Tr(K_T^k),
    \qquad
    k \in \mN
\end{equation}
(the latter is obtained by comparing the right-hand sides of (\ref{expquad}), (\ref{expquad1})).
A combination of (\ref{KTasy}) with (\ref{bCK}) establishes (\ref{bCasy}). 
\end{proof}

The relation (\ref{bCasy}) clarifies the probabilistic meaning of the Hardy-Schatten norms
of the transfer function in the context of the asymptotic behaviour of the output energy cumulants.  As shown below, the covariance-analytic cost (\ref{phiFF}) plays a similar role for the following functional of the covariance operator (\ref{KT}):
\begin{equation}
\label{TrphiKT}
    \Tr \phi(K_T)
    =
    \sum_{k=1}^{+\infty}
    \phi_k
    \Tr (K_T^k)
    =
    \sum_{k=1}^{+\infty}
    \frac{1}{(2k-2)!!}
    \phi_k
    C_{k, T},
\end{equation}
where the last equality uses (\ref{bCK}). Such ``trace-analytic'' functionals arise in quantum mechanics \cite{L_1973} and are also used in a control theoretic context (in application to matrices rather than integral operators); see, for example,  \cite{SIG_1998} and \cite[Section X]{VP_2010a}. An equivalent representation of (\ref{TrphiKT}) in terms of the resolvent of the operator $K_T$  is given by
$$    \Tr \phi(K_T)
    =
    \frac{1}{2\pi i}
    \Tr
    \oint
    \phi(z)
    (z\cI - K_T)^{-1}
    \rd z,
$$
where the integral is over any  counterclockwise  oriented closed contour in the analyticity domain of $\phi$ around the interval $[0, \|F\|_{\infty}^2]$ containing the spectrum of $K_T$. In view of (\ref{CPi}), (\ref{Pi}),  the quantity (\ref{TrphiKT}) is a complicated function  of the output energy moments:
$$
    \Tr \phi(K_T)
    =
    \sum_{k=1}^{+\infty}
    \frac{1}{(2k-2)!!}
    \phi_k
    \sum_{\ell_1, \ldots, \ell_k\> 0:\ \sum_{j=1}^k j\ell_j = k}
    \pi_{k; \ell_1, \ldots, \ell_k}
    \prod_{j=1}^k
    (\bE (\cE_T^j))^{\ell_j}.
$$

\begin{theorem}
\label{th:covan}
Suppose the assumptions of Theorem~\ref{th:CK} are satisfied.
Then for any function $\phi$, analytic in a neighbourhood of $[0,\|F\|_\infty^2]$ and satisfying (\ref{phi0}), the growth rate of the functional (\ref{TrphiKT}) coincides with (\ref{phiFF}):
\begin{equation}
\label{phiKTasy}
    \lim_{T\to +\infty}
    \Big(
    \frac{1}{T}
    \Tr \phi(K_T)
    \Big)
    =
    J_\phi(F).
\end{equation}
\hfill$\square$
\end{theorem}
\begin{proof}
Since the function $\phi$ is analytic in a neighbourhood of $[0,\|F\|_\infty^2]$, then the radius of convergence of its Taylor series at the origin is strictly greater than $\|F\|_\infty^2$, and hence,
\begin{equation}
\label{major}
    \sum_{k \in \mN}
    |\phi_k|
    \|F\|_\infty^{2k}
    < +\infty.
\end{equation}
Similarly to (\ref{SSS}), the inequality (\ref{KTnorm}) and the positive semi-definiteness of the operator $K_T$ lead to
\begin{equation}
\label{KKK}
    \Tr (K_T^{k+1})
    =
    \Tr (K_T^{k/2} K_T K_T^{k/2})
    \<
    \|F\|_\infty^2 \Tr (K_T^k)
\end{equation}
for any $k \in \mN$. Therefore, by induction on $k$, combined with (\ref{Ktrace}), it follows from (\ref{KKK}) that
\begin{equation}
\label{Ktrace1}
    \Tr(K_T^k)
    \< \|F\|_\infty^{2k-2} \Tr K_T = T\|F\|_\infty^{2k-2} \|F\|_2^2,
    \qquad
    k \in \mN.
\end{equation}
A combination of (\ref{KTasy}), (\ref{major}), (\ref{Ktrace1}) along with a dominated convergence argument (using the upper bound $|\phi_k \Tr (K_T^k)/T| \< (\|F\|_2/\|F\|_\infty)^2|\phi_k|  \|F\|_\infty^{2k}$ which holds for all $k\in \mN$ uniformly over $T>0$) yields
$$    \frac{1}{T}
    \Tr \phi(K_T)
    =
    \sum_{k=1}^{+\infty}
    \phi_k
    \frac{1}{T}
    \Tr (K_T^k)
    \to
    \sum_{k=1}^{+\infty}
    \phi_k
    \|F\|_{2k}^{2k},
$$
as $T\to +\infty$, thus establishing (\ref{phiKTasy}) in view of (\ref{JHS}).
\end{proof}

\section{A variational inequality with covariance-analytic functionals}
\label{sec:var}

Suppose the input $W$ of the system (\ref{XZ}) is a Gaussian process in $\mR^m$ with stationary increments (not necessarily independent as in the standard Wiener process case),  whose covariance structure  is specified by
\begin{equation}
\label{Wcov}
    \bE
    \Big(
        \int_\mR g(s)^\rT \rd W(s)
        \int_\mR h(t)^\rT \rd W(t)
    \Big)
    =
    \bra
        g,
        D(h)
    \ket,
\end{equation}
where $D$ is a positive semi-definite  self-adjoint covariance operator on the Hilbert space $L^2(\mR, \mR^m)$. 
The operator $D$ is shift-invariant in the sense of the commutativity $D\cT_\tau = \cT_\tau D$ with the translation operator $\cT_\tau$ acting on a function $\psi : \mR\to \mR^m$ as $(\cT_\tau \psi)(t)= \psi(t+\tau)$ for any $t, \tau \in \mR$. Also, suppose the covariance operator $D$ has a spectral density $\Sigma: \mR\to \mH_m^+$ which relates (\ref{Wcov}) to the Fourier transforms $\wh{g}$, $\wh{h}$ of the functions $g$, $h$ as
\begin{equation}
\label{gDh}
        \bra
        g,
        D(h)
    \ket
    =
    \frac{1}{2\pi}
    \int_{\mR}
    \wh{g}(\omega)^*
    \Sigma(\omega)
    \wh{h}(\omega)
    \rd \omega .
\end{equation}
For what follows, it is assumed that the input spectral density $\Sigma$ admits the factorisation
\begin{equation}
\label{GG}
  \Sigma(\omega) = \Gamma(i\omega) \Gamma(i\omega)^*,
  \qquad
  \omega \in \mR,
\end{equation}
where $\Gamma \in \cH_\infty^{m\x m}$ is the transfer function of a shaping filter 
which produces $W$ from a standard Wiener process $V:=(V(t))_{t \in \mR}$
in $\mR^m$
according to an appropriately modified form of (\ref{ZfW}); see Fig.~\ref{fig:FG}.
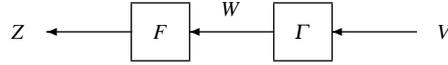
\begin{figure}[htbp]
\unitlength=0.75mm
\linethickness{0.5pt}
\centering
\begin{picture}(100,15.00)
    \put(35,2){\framebox(10,10)[cc]{{$F$}}}
    \put(60,2){\framebox(10,10)[cc]{{$\Gamma$}}}
    \put(35,7){\vector(-1,0){15}}
    \put(60,7){\vector(-1,0){15}}
    \put(85,7){\vector(-1,0){15}}
    \put(15,7){\makebox(0,0)[cc]{$Z$}}
    \put(90,7){\makebox(0,0)[cc]{$V$}}
    \put(52.5,11){\makebox(0,0)[cc]{$W$}}
\end{picture}
\caption{The linear stochastic system $F$ with the output $Z$ and input $W$ produced by a shaping filter $\Gamma$ from a standard Wiener process $V$.}
\label{fig:FG}
\end{figure}
In particular, if $W$ is a standard Wiener process in $\mR^m$, then it has the identity covariance operator $D=\cI$ with the constant spectral density $\Sigma(\omega) = I_m$, in which case the shaping filter is an all-pass system, with the matrix $\Gamma(i\omega)$ being unitary  for almost all $\omega \in \mR$.  Returning to the more general case in (\ref{Wcov}), (\ref{gDh}),  the output spectral density (\ref{S}) is modified as
\begin{equation*}
\label{S1}
    S(\omega)
    :=
    \wh{F}(\omega) \Sigma(\omega)\wh{F}(\omega)^*,
\end{equation*}
thus allowing the variance of the stationary output process $Z$ in (\ref{ZfW}) to be represented in this case as
\begin{equation}
\label{Zvar}
    \bE(|Z(t)|^2)
     =
     \|F\Gamma\|_2^2
     =
    \frac{1}{2\pi}
    \int_{\mR}
    \Tr S(\omega)
    \rd \omega
    =
    \frac{1}{2\pi}
    \int_{\mR}
    \bra \Lambda(\omega), \Sigma(\omega)\ket_\rF \rd \omega,
\end{equation}
where $\Lambda: \mR\to \mH_m^+$ is an auxiliary function associated with the Fourier transform (\ref{Fhat}) by
\begin{equation}
\label{FF}
  \Lambda(\omega)
  :=
  \wh{F}(\omega)^* \wh{F}(\omega).
\end{equation}
Up to $|p-m|$ zero eigenvalues, the matrix $\Lambda(\omega)$ is isospectral to $\wh{F}(\omega) \wh{F}(\omega)^*$ on the right-hand side of (\ref{S}), so that the $\cH_{2k}$-norms in (\ref{H2k}) can be expressed in terms of (\ref{FF}) as
\begin{equation*}
\label{H2k1}
    \|F\|_{2k}
     :=
    \sqrt[2k]{
    \frac{1}{2\pi}
    \int_{\mR}
    \Tr(\Lambda(\omega)^k)
    \rd \omega},
    \qquad
    k \in \mN.
\end{equation*}
In particular, application of the Cauchy–Bunyakovsky–Schwarz inequality to the right-hand side of (\ref{Zvar}) leads to an upper bound for the root mean square value of the output process:
\begin{align}
\nonumber
    \|F\Gamma\|_2
    & =
    \sqrt{\bE(|Z(t)|^2)}
    \\
\label{Zvar1}
    & \<
    \frac{1}{\sqrt{2\pi}}
    \sqrt[4]{
    \int_{\mR}
    \|\Lambda(\omega)\|_\rF^2
    \rd \omega
    \int_{\mR}
    \|\Sigma(\omega)\|_\rF^2
    \rd \omega}
    =
    \|F\|_4 \|\Gamma\|_4,
\end{align}
where
\begin{equation*}
\label{GH4}
    \|\Gamma\|_4
    =
    \sqrt[4]{
    \frac{1}{2\pi}
    \int_{\mR}
    \|\Sigma(\omega)\|_\rF^2
    \rd \omega}
\end{equation*}
is the $\cH_4$-norm of the transfer function $\Gamma$ of the input shaping filter, in accordance with (\ref{H4}), (\ref{GG}). The inequality (\ref{Zvar1}) also yields an upper bound on the induced norm of the system as a linear operator acting from $\cH_4^{m\x m}$ to $\cH_2^{p\x m}$:
$$
    \sup_{\Gamma \in \cH_4^{m\x m}\setminus \{0\}}
    \frac{\|F\Gamma\|_2}{\|\Gamma\|_4}
    \<
    \|F\|_4,
$$
which illustrates the role of the $\cH_4$-norms for the operator theoretic properties of linear systems.

We will now discuss a variational inequality involving the covariance-analytic functionals    (\ref{phiFF}). To this end, consider the Legendre transformation \cite{R_1970} of a trace-analytic function  on Hermitian matrices whose eigenvalues are contained by an interval $\Delta:= (a, b )$ (with possibly infinite endpoints $a<b$):
\begin{equation}
\label{fL}
    \fL(\psi)(L)
    :=
    \sup_{M \in \mH_m:\ a I_m \prec M\prec bI_m}
    (
    \bra
        L, M
    \ket_\rF - \Tr \psi(M)),
    \qquad
    L \in \mH_m,
\end{equation}
where $\psi$ is an analytic function in a neighbourhood  of $\Delta$ in the complex plane with real values on this interval.
The following lemma establishes invariance of a class of such functions under the Legendre transformation.

\begin{lemma}
\label{lem:LT}
Suppose the function $\psi$, described in the context of (\ref{fL}), has a positive second derivative:
\begin{equation}
\label{psi''}
    \psi''(z)>0,
    \qquad
    z \in \Delta.
\end{equation}
Then the Legendre transformation of $\Tr\psi$ yields a trace-analytic function, representable as
\begin{equation}
\label{Lpsi}
        \fL(\psi)(L)
        =
        \Tr \psi_*(L)
\end{equation}
for any matrix $L\in \mH_m$ with eigenvalues in the interval
$    \psi'(\Delta) =(\psi'(a+), \psi'(b-))$
whose endpoints are the corresponding one-sided limits of the first derivative $\psi'$. Here,
\begin{equation}
\label{psi*}
        \psi_*(z) := z v(z) - \psi(v(z))
\end{equation}
is the Legendre transformation of $\psi$ as a function of a real variable,
with $v:\psi'(\Delta)\to \Delta$ the functional inverse of $\psi'$:
\begin{equation}
\label{v}
  v:= (\psi')^{-1}.
\end{equation}
The function $\psi_*$ in (\ref{psi*}) is analytic in a neighbourhood of $\psi'(\Delta)$ in the complex plane  and satisfies
\begin{equation}
\label{psi*''}
    \psi_*''(z)>0,
    \qquad
    z \in \psi'(\Delta).
\end{equation}
\hfill$\square$
\end{lemma}
\begin{proof}
In view of \cite[Lemma 4, Appendix A]{VP_2010a} (see also \cite[p. 270]{SIG_1998}), the Frechet derivative of $\Tr \psi$  on the real Hilbert space $\mH_m$ is related to the derivative  $\psi'$ (as a function of a complex variable) by
\begin{equation}
\label{psi'}
    \d_M \Tr \psi(M) = \psi'(M),
    \qquad
    M \in \mH_m,
    \
    a I_m \prec M \prec b I_m.
\end{equation}
In combination with the identity $\d_M \bra L, M \ket_\rF = L$ for any matrix $L \in \mH_m$, the relation (\ref{psi'}) leads to
\begin{equation}
\label{diff}
    \d_M
    (
        \bra
        L, M
    \ket_\rF - \Tr \psi(M)
    ) = L - \psi'(M).
\end{equation}
The condition (\ref{psi''}) implies that  $\psi'$ is strictly increasing on $\Delta$ and, therefore,  has a functional inverse $v: \psi'(\Delta)\to \Delta$ in (\ref{v}), which is differentiable on the interval $\psi'(\Delta)$, with
\begin{equation}
\label{''}
    v'(z) = \frac{1}{\psi''(v(z))} = \psi_*''(z) > 0,
    \qquad
    z \in \psi'(\Delta)  ,
\end{equation}
thus proving (\ref{psi*''}) for the function $\psi_*$ in (\ref{psi*}), with the second  equality in (\ref{''}) following   from the properties of the Legendre transformation \cite{R_1970}.
Moreover, since $\psi$ is analytic  in a neighbourhood of $\Delta$, then by applying the analytic implicit function theorem (see, for example, \cite[Theorem 2.2.8 on p. 24]{H_1990}), it follows that $v$ in (\ref{v}) admits an analytic continuation to a neighbourhood of the interval $\psi'(\Delta)$ in the complex plane, and hence, so also does $\psi_*$ in (\ref{psi*}).
Now, the strict convexity of $\psi$ on $\Delta$ is inherited by $\Tr \psi(M)$ as a function of $M \in \mH_m$ such that $a I_m \prec M \prec b I_m$. Hence, for any $L \in \mH_m$ satisfying $\psi'(a+) I_m \prec L \prec \psi'(b-) I_m$, the strictly concave function $M \mapsto \bra
        L, M
    \ket_\rF - \Tr \psi(M)$ achieves its supremum in (\ref{fL}) at a unique  matrix $M\in \mH_m$, with $a I_m \prec M \prec b I_m$, at
which the Frechet derivative (\ref{diff}) vanishes, so that $\psi'(M) = L$ or, equivalently,
\begin{equation}
\label{argmax}
    \mathop{\mathrm{arg\, max}}_{M \in \mH_m:\ a I_m \prec M \prec b I_m}
    (
    \bra
        L, M
    \ket_\rF - \Tr \psi(M))
    =
    v(L).
\end{equation}
By substituting the right-hand side of (\ref{argmax}) for $M$, the Legendre transformation  in (\ref{fL}) is represented as
$$
    \fL(\psi)(L)
     =
    \bra
        L, v(L)
    \ket_\rF
    -
    \Tr \psi(v(L))
    =
    \Tr
    (
        L v(L)
    -
    \psi(v(L))),
$$
which establishes (\ref{Lpsi}).
\end{proof}

The functions $\psi: \Delta \to \mR$ and $\psi_*: \psi'(\Delta)\to \mR$
in Lemma~\ref{lem:LT} are strictly convex and form a convex conjugate pair (related by the Legendre transformation for functions of a real variable).

Now, suppose the input spectral density $\Sigma$ in (\ref{GG}) is known imprecisely, and the statistical uncertainty is described by
\begin{equation}
\label{fG}
    \fG_{\psi,d}
    :=
    \{
        \Gamma
        \in
        \cH_\infty^{m\x m}:
        \
        J_\psi(\Gamma) \< d
    \}.
\end{equation}
Here, $\psi$ is a given function, which takes nonnegative values on the positive half-line $(0,+\infty)$ and is analytic in its neighbourhood in the complex plane. The corresponding
covariance-analytic functional $J_\psi$ is evaluated at the transfer function $\Gamma$ of the input shaping filter as
\begin{equation*}
\label{psiG}
    J_\psi(\Gamma)
    :=
    \frac{1}{2\pi}
  \int_{\mR}
  \Tr \psi(\Sigma(\omega))
  \rd \omega,
\end{equation*}
in accordance with (\ref{phiFF}), (\ref{GG}), for which it is assumed that $\det \Gamma(i\omega) \ne 0$ (that is, $\Sigma(\omega)\succ 0$) at almost all frequencies $\omega \in \mR$. 
While $\psi$ reflects the general structure of the uncertainty class $\fG_{\psi,d}$ in (\ref{fG}), its ``size''  is measured by a parameter $d\> 0$.  The following theorem discusses the worst-case output variance in the framework of this uncertainty description.

\begin{theorem}
\label{th:supvar}
Suppose the statistical uncertainty about the input $W$ is described by (\ref{fG}), where the function $\psi$,  with nonnegative values on $\Delta:= (0,+\infty)$ and analytic in its neighbourhood in the complex plane, satisfies (\ref{psi''}). Then for any $d\> 0$, the worst-case value of the output variance for the system $F$ in (\ref{Zvar}) admits a guaranteed upper bound
\begin{equation}
\label{supvar}
    \sup_{\Gamma \in \fG_{\psi,d}}
    \bE(|Z(t)|^2)
    \<
    \inf_{\sigma > 0}
    \Big(
    \frac{1}{\sigma}
    (
        J_{\psi_*}(\sqrt{\sigma}F)
        +
        d
    )
    \Big)  ,
\end{equation}
where $\psi_*$ is the convex conjugate of $\psi$, given by (\ref{psi*}), (\ref{v}) of Lemma~\ref{lem:LT}. \hfill$\square$
\end{theorem}
\begin{proof}
For admissible matrices $L, M \in \mH_m$ in (\ref{fL}), (\ref{Lpsi}) of Lemma~\ref{lem:LT}, the Fenchel inequality takes the form
\begin{equation}
\label{LM}
    \bra L, M\ket_\rF \< \Tr (\psi_*(L) + \psi(M)),
\end{equation}
which is extended to other Hermitian matrices $L, M$ by letting its right-hand side equal $+\infty$.
Multiplication of $L$ by a scalar $\sigma >0$ and division of both sides of (\ref{LM}) by $\sigma$ yields
\begin{equation}
\label{LM0}
    \bra L, M\ket_\rF
    =
    \frac{1}{\sigma}
    \bra \sigma L, M\ket_\rF
    \<
    \frac{1}{\sigma}\Tr (\psi_*(\sigma L) + \psi(M)).
\end{equation}
Since the left-hand side of (\ref{LM0}) does not depend on $\sigma$, this inequality can be tightened as
\begin{equation}
\label{LM1}
    \bra L, M\ket_\rF
    \<
    \inf_{\sigma >0}\Big(\frac{1}{\sigma}\Tr (\psi_*(\sigma L) + \psi(M))\Big).
\end{equation}
Application of (\ref{LM}) to the matrices $L:= \Lambda(\omega)$ and $M:= \Sigma(\omega)$  in (\ref{FF}), (\ref{GG}), integration over the frequencies $\omega \in \mR$ in accordance with  (\ref{Zvar}) and the scaling argument as in  (\ref{LM1})   lead to
\begin{align}
\nonumber
    \|F\Gamma\|_2^2
     & \<
    \frac{1}{2\pi}
    \inf_{\sigma > 0}
    \Big(
    \frac{1}{\sigma}
    \int_{\mR}
    \Tr (\psi_*(\sigma \Lambda(\omega)) + \psi(\Sigma))
    \rd \omega
    \Big)\\
\label{JJ1}
    & =
    \inf_{\sigma > 0}
    \Big(
    \frac{1}{\sigma}
    (
        J_{\psi_*}(\sqrt{\sigma}F)
        +
        J_\psi(\Gamma)
    )
    \Big)  .
\end{align}
Here, the covariance-analytic functionals $J_\psi$, $J_{\psi_*}$, associated with the function $\psi$ and its convex conjugate $\psi_*$  in (\ref{psi*}), are evaluated at the transfer functions $\Gamma$ and $\sqrt{\sigma}F$ of the input shaping filter and the rescaled  system, respectively; see Fig.~\ref{fig:FG}.  Since the right-hand side of (\ref{JJ1}) depends on $J_\psi(\Gamma)$ in a monotonic fashion, and $J_\psi(\Gamma) \< d$ for any $\Gamma \in \fG_{\psi, d}$ in the uncertainty class (\ref{fG}), then
$$
    \sup_{\Gamma \in \fG_{\psi,d}}
    \|F\Gamma\|_2^2
    \<
    \inf_{\sigma > 0}
    \Big(
    \frac{1}{\sigma}
    (
        J_{\psi_*}(\sqrt{\sigma}F)
        +
        d
    )
    \Big)  ,
$$
which establishes (\ref{supvar}).
\end{proof}

For any given $\sigma>0$, the right-hand side of (\ref{supvar}) depends monotonically on the quantity
$$
    J_{\psi_*}(\sqrt{\sigma} F) = J_\phi (F),
$$
where the function
\begin{equation}
\label{phipsi}
    \phi(z) := \psi_*(\sigma z)
\end{equation}
is related to $\psi_*$ in (\ref{psi*}). In a robust control setting, where the transfer function $F$ pertains to the closed-loop system,  this monotonicity justifies the minimisation of the covariance-analytic cost (\ref{phiFF}) as a way to guarantee an improved bound on the worst-case output variance in (\ref{supvar}). This approach is similar to minimax LQG control \cite{PJD_2000}, except that here the relative entropy description of uncertainty  is replaced with (\ref{fG}), while the risk-sensitive cost (\ref{QE}) is replaced with (\ref{phiFF}).

The uncertainty description (\ref{fG})  and the corresponding cost (\ref{phiFF}), specified by (\ref{phipsi}),  employ the covariance-analytic functionals $J_\psi$, $J_{\psi_*}$. Their link through the Legendre transformation (\ref{psi*}) resembles (and, in some respects, generalises)  the duality relation \cite{DE_1997} which underlies the role of the risk-sensitive control for robustness properties  of systems \cite{DJP_2000}  in minimax LQG settings.

Although (\ref{fG})  is concerned only with the input spectral densities irrespective of entropy constructs, we note that the function $\phi_\theta$ in (\ref{QEphi}), associated with the risk-sensitive cost $\Xi(\theta)$ in  (\ref{QE}), leads to $$
    \psi_*(z) = \phi_\theta(z/\sigma) = -\frac{1}{2}\ln(1-2z),
    \qquad
    z < \frac{1}{2},
$$
in (\ref{phipsi}), considered  with the scaling $\theta=2 \sigma$ for convenience.  This function $\psi_*$ is the convex conjugate of
\begin{equation*}
\label{psit}
    \psi(z):= \frac{1}{2}(z - 1-\ln z),
    \qquad
    z>0,
\end{equation*}
which is a strictly convex nonnegative function  (admitting an analytic continuation to a neighbourhood of $(0,+\infty)$ in the complex plane),  with
$$
    \psi(1) = \min_{z >0} \psi(z) = 0.
$$
The corresponding uncertainty class $\fG_{\psi, d}$ in (\ref{fG}) takes into account the deviation of the input spectral densities $\Sigma$ in (\ref{GG}) from the identity matrix  $I_m$, which is the only element of $\fG_{\psi,0}$.

\section{Strictly proper finite-dimensional systems}
\label{sec:fin}

We will now proceed to the state-space computation of the covariance-analytic functionals (\ref{phiFF}) (which, in essence,   reduces to that of the Hardy-Schatten norms (\ref{H2k})) for finite-dimensional systems. More precisely, consider a class of  systems (\ref{ZfW}) with an $\mR^n$-valued stationary Gaussian state process $X:= (X(t))_{t\in \mR}$ governed by an Ito SDE as
\begin{equation}
\label{XZ}
  \rd X = A X \rd t + B\rd W,
  \qquad
  Z= CX,
\end{equation}
where $A\in \mR^{n\x n}$, $B \in \mR^{n\x m}$, $C \in \mR^{p\x n}$ are given matrices, with $A$ Hurwitz, thus making the impulse response
\begin{equation}
\label{fABC}
    f(t)
    =
    C\re^{tA}B,
    \qquad
    t \>0
\end{equation}
square integrable. This system has a strictly proper real-rational transfer function (\ref{Ff}):
\begin{equation}
\label{F0}
    F(s)
    :=
    C E(s)B,
    \qquad
    s \in \mC,
\end{equation}
where
\begin{equation}
\label{Es}
    E(s)
    :=
    (sI_n - A)^{-1}
\end{equation}
is an auxiliary $\mC^{n\x n}$-valued transfer function from the increments of the process $BW$ to the internal state $X$ of the system in (\ref{XZ}). The state-space realisations of $F$, $E$ are denoted by
\begin{equation}
\label{Ereal}
    F
    =
      \left[
    \begin{array}{c|c}
    A & B\\
      \hline
      C &  0
    \end{array}
    \right],
    \qquad
    E
    =
      \left[
    \begin{array}{c|c}
    A & I_n\\
      \hline
      I_n &  0
    \end{array}
    \right],
\end{equation}
so that the corresponding system conjugates
\begin{align}
\label{Fsim}
    F^\sim(s)
    & :=
    F(-\overline{s})^*
    =
    B^\rT E^\sim(s) C^\rT,\\
\label{Esim}
    E^\sim(s)
    & :=
    E(-\overline{s})^*
    =
    -(sI_n + A^\rT)^{-1}
\end{align}
are represented as
\begin{equation}
\label{Esimreal}
    F^\sim
    =
      \left[
    \begin{array}{c|c}
    -A^\rT & C^\rT\\
      \hline
      -B^\rT &  0
    \end{array}
    \right],
    \qquad
    E^\sim
    =
      \left[
    \begin{array}{c|c}
    -A^\rT & I_n\\
      \hline
      -I_n &  0
    \end{array}
    \right].
\end{equation}
Since the matrix $A$ is Hurwitz, the transfer function $F$ in (\ref{F0}) satisfies (\ref{Fgood1}) along with its corollary (\ref{Fgood}), whereby it has finite Hardy-Schatten norms $\|F\|_{2k}$ of any order in (\ref{H2k}).

In the finite-dimensional case being considered, the spectral density $S$  and its powers in (\ref{H2k})  are rational functions which can be    identified with transfer functions of such systems. The state-space realisations of the systems $F$, $E$ and their conjugates $F^\sim$, $E^\sim$ in (\ref{Ereal})--(\ref{Esimreal}) allow the spectral density $S$ in (\ref{S}) (with a slight abuse of notation,  as a function of $i\omega$ rather than the frequency $\omega$ itself) to be realised as the transfer function of the system
\begin{equation}
\label{Sreal}
    S
    =
    F F^\sim
    =
    CE \mho E^\sim C^\rT
    =
    \left[
    \begin{array}{cc|c}
      -A^\rT & 0 & C^\rT\\
      -\mho & A & 0\\
      \hline
      0 & C & 0
    \end{array}
    \right],
\end{equation}
where
\begin{equation}
\label{BB}
    \mho:= BB^\rT.
\end{equation}
In view of the identity
$$
  \begin{bmatrix}
    I_n & 0 \\
    \alpha & I_n
  \end{bmatrix}
  \begin{bmatrix}
    I_n & 0 \\
    \beta & I_n
  \end{bmatrix}
  =
  \begin{bmatrix}
    I_n & 0 \\
    \alpha + \beta & I_n
  \end{bmatrix},
  \qquad
  \alpha, \beta \in \mC^{n\x n} ,
$$
the dynamics matrix of the state-space realisation (\ref{Sreal}) is block diagonalised by the similarity transformation
\begin{align}
\nonumber
  \begin{bmatrix}
    I_n & 0 \\
    P & I_n
  \end{bmatrix}
  \begin{bmatrix}
    -A^\rT & 0\\
    -\mho & A
  \end{bmatrix}
  \begin{bmatrix}
    I_n & 0 \\
    -P & I_n
  \end{bmatrix}
    & =
  \begin{bmatrix}
    -A^\rT & 0\\
    -P A^\rT-\mho & A
  \end{bmatrix}
  \begin{bmatrix}
    I_n & 0 \\
    -P & I_n
  \end{bmatrix}\\
\label{block2}
    & =
  \begin{bmatrix}
    -A^\rT & 0\\
    A P & A
  \end{bmatrix}
  \begin{bmatrix}
    I_n & 0 \\
    -P & I_n
  \end{bmatrix}=
  \begin{bmatrix}
    -A^\rT & 0\\
    0 & A
  \end{bmatrix},
\end{align}
which employs the covariance matrix of the stationary state process $X$ in (\ref{XZ}),  related by
\begin{equation}
\label{P}
  P
   :=
  \bE(X(t)X(t)^\rT)
  =
  \frac{1}{2\pi}
  \int_{\mR}
  E(i\omega) \mho E(i\omega)^*\rd \omega
    =
    \cL_A(\mho)
\end{equation}
to the transfer function $E$ in (\ref{Es}) and the matrix $\mho$ from (\ref{BB}),
and satisfying the ALE
\begin{equation}
\label{PALE}
  AP + PA^\rT + \mho  = 0,
\end{equation}
so that $P$ coincides with the controllability Gramian of the pair $(A,B)$.
The right-hand side of (\ref{P}) involves a linear operator $\cL_A$ on $\mC^{n\x n}$, associated with the Hurwitz matrix $A$ as
\begin{equation}
\label{cL}
  \cL_A(V) : = \int_{\mR_+} \re^{tA} V \re^{tA^\rT}\rd t,
  \qquad
  V \in \mC^{n\x n}.
\end{equation}
Due to the commutativity between $\cL_A$ and the matrix transpose ($\cL_A(V^\rT) = \cL_A(V)^\rT$ for any $V \in \mR^{n\x n}$),  the subspace $\mS_n$ of real symmetric matrices of order $n$ is invariant under $\cL_A$:
\begin{equation}
\label{incs}
    \cL_A(\mS_n) \subset \mS_n.
\end{equation}
Also, since $A$ is Hurwitz,   the inclusion (\ref{incs}) is complemented by the invariance of the set $\mS_n^+$
of real positive semi-definite symmetric matrices of order $n$ under the operator $\cL_A$:
\begin{equation}
\label{inc2}
    \cL_A(\mS_n^+) \subset \mS_n^+.
\end{equation}
Therefore,  in view of (\ref{block2}), the system (\ref{Sreal})  admits an equivalent state-space realisation:
\begin{equation}
\label{Sreal1}
    S
    =
    \left[
    \begin{array}{c|c}
      a & b\\
      \hline
      c & 0
    \end{array}
    \right],
    \quad
        a
    :=
    \begin{bmatrix}
        -A^\rT & 0\\
  0 & A
    \end{bmatrix},
    \quad
    b:=
    \begin{bmatrix}
      C^\rT\\
      PC^\rT
    \end{bmatrix},
    \quad
    c:=
    \begin{bmatrix}
      -CP  & \quad C
    \end{bmatrix}.
\end{equation}

\section{Infinite cascade spectral factorization}
\label{sec:inf}

With the spectral density $S$ in (\ref{S}) being represented as the transfer function of a finite-dimensional system with the strictly proper   state-space realisation (\ref{Sreal}) (or (\ref{Sreal1})),
the matrix $\phi(S)$ in (\ref{phiFF}) corresponds to the transfer  function for an infinite cascade of such systems shown in Fig.~\ref{fig:phi}.
\begin{figure}[htbp]
\centering
\unitlength=0.75mm
\linethickness{0.5pt}
\begin{picture}(55.00,41.00)
    \put(-15,2){\dashbox(80,37)[cc]{}}
    \put(-14,38){\makebox(0,0)[lt]{{$\phi(S)$}}}
    \put(11,28){\framebox(8,8)[cc]{$S$}}
    \put(31,28){\framebox(8,8)[cc]{$S$}}
    \put(51,28){\framebox(8,8)[cc]{$S$}}

    \put(11,32){\vector(-1,0){12}}
    \put(51,32){\vector(-1,0){12}}
    \put(31,32){\vector(-1,0){12}}
    \put(71,32){\vector(-1,0){12}}
    \put(45,32){\vector(0,-1){8}}
    \put(25,32){\vector(0,-1){8}}
    \put(5,32){\vector(0,-1){8}}
    \put(41,16){\framebox(8,8)[cc]{$\phi_1$}}
    \put(21,16){\framebox(8,8)[cc]{$\phi_2$}}
    \put(1,16){\framebox(8,8)[cc]{$\phi_3$}}
    \put(45,16){\line(0,-1){9}}
    \put(25,16){\vector(0,-1){6.3}}
    \put(5,16){\vector(0,-1){6.3}}
    \put(25,7){\circle{5}}
    \put(5,7){\circle{5}}
    \put(45,7){\vector(-1,0){17.3}}
    \put(22.3,7){\vector(-1,0){14.5}}
    \put(-9,7){\vector(-1,0){12}}
    \put(25,7){\makebox(0,0)[cc]{{$+$}}}
    \put(5,7){\makebox(0,0)[cc]{{$+$}}}
    \put(-2,32){\makebox(0,0)[rc]{{$\cdots$}}}
    \put(-2,7){\makebox(0,0)[rc]{{$\cdots$}}}
    \put(73,32){\makebox(0,0)[lc]{{$\ups$}}}
    \put(-24,7){\makebox(0,0)[lc]{{$\zeta$}}}
\end{picture}
\caption{An infinite cascade of identical linear systems with the common state-space realization of $S$ in  (\ref{Sreal}), (\ref{Sreal1}) and an external input $\ups$. The sum of their outputs, weighted by the coefficients  $\phi_k$  of the Taylor series expansion of the analytic  function $\phi$ from (\ref{phiFF}) in accordance with (\ref{JHS}), forms the output $\zeta$ of the system $\phi(S)$.
}
\label{fig:phi}
\end{figure}
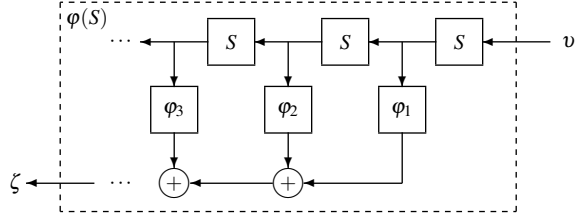
The resulting system  has an infinite dimensional state and the strictly proper state-space realisation
\begin{equation}
\label{phitrans}
    \phi(S)
    =
    \left[\begin{array}{cccc|c}
  a   & 0 & 0 &\ldots & b\\
  b c & a & 0 & \ldots & 0\\
  0   & bc & a & \ldots & 0\\
  \ldots    &\ldots   & \ldots  & \ldots & \ldots  \\
  \hline
  \phi_1  c & \phi_2  c & \phi_3 c & \cdots & 0
\end{array}\right],
\end{equation}
where $a$, $b$, $c$
are the state-space matrices of $S$ in (\ref{Sreal1}). In accordance with its cascade structure, the system  (\ref{phitrans}) has a  block  two-diagonal   lower triangular dynamics matrix. As established below, this system admits an inner-outer  factorization (see, for example, \cite{W_1972}) with a similar infinite cascade structure. This result is based on the following two lemmas and a theorem, which adapt a ``system transposition'' technique from \cite{VP_2022}.

\begin{lemma}
\label{lem:EOE}
Suppose the matrix $A \in \mR^{n\x n}$ is Hurwitz, and $U=U^\rT \in \mR^{n\x n}$ is a positive semi-definite matrix such that the pair $(A,\sqrt{U})$ is controllable.
Then the transfer function $E$ in (\ref{Es}) and its system conjugate $E^\sim$ in (\ref{Esim}) satisfy
\begin{equation}
\label{EOE}
    E(s) U E^\sim(s) = V E^\sim(s)  V^{-1}U V^{-1} E(s)V
\end{equation}
for any $s\in \mC$ beyond the spectra of $\pm A$, where $V$ is the unique solution of the ALE
\begin{equation}
\label{VALE}
    AV + VA^\rT + U = 0.
\end{equation}
\hfill$\square$
\end{lemma}
\begin{proof}
The relation (\ref{EOE}) is a corollary of \cite[Lemma 2]{VP_2022}. The condition  $U\succcurlyeq 0$ and the controllability of $(A,\sqrt{U})$ ensure that the solution of (\ref{VALE}) satisfies $V \succ 0$ (and is, therefore, nonsingular).
\end{proof}

In application of Lemma~\ref{lem:EOE} to the matrix $U:= \mho$ from (\ref{BB}), the ALE (\ref{VALE}) coincides with (\ref{PALE}) and yields the covariance matrix $V:= P$ in (\ref{P}). Assuming that the pair $(A,B)$ is controllable (and hence, $P\succ 0$), the relation  (\ref{EOE}) allows the factors $E$ and $E^\sim$  in (\ref{Sreal}) to be rearranged as
\begin{equation}
\label{EmhoE}
  E \mho E^\sim = P E^\sim  P^{-1}\mho P^{-1} EP ,
\end{equation}
thus leading to
\begin{equation}
\label{SEE}
  S = CP E^\sim  P^{-1}\mho P^{-1} EPC^\rT ,
\end{equation}
so that the factor $E$ is moved to the right, while its dual $E^\sim$ is moved to the left. This rearrangement resembles the Wick ordering for mixed products of noncommuting annihilation and creation operators in quantum mechanics (see \cite{W_1950} and \cite[pp. 209--210]{J_1997}).   Theorem~\ref{th:EOEk} below uses (\ref{EmhoE}) in order to extend (\ref{SEE}) to arbitrary positive integer powers of $S$. Its formulation employs three sequences of matrices $\alpha_j, \beta_j, \gamma_j  \in \mR^{n\x n}$ computed recursively as
\begin{align}
\label{alf}
    \alpha_{j+1}
    & = \gamma_j \beta_j,\\
\label{bet}
    \beta_{j+1}
    & =
    \gamma_j^{-1} \alpha_j \Omega^{\delta_{j1}}\gamma_{j-1}\gamma_j^{-1},\\
\label{gam}
    \gamma_j
    & =
    \cL_A(\alpha_j \Omega^{\delta_{j1}}\gamma_{j-1}),
    \qquad
    j\in \mN
\end{align}
(where $\delta_{jk}$ is the Kronecket delta, and  the operator $\cL_A$ is given by (\ref{cL})),
with the initial conditions
\begin{equation}
\label{alfbetgam0}
        \alpha_1
    =
    \gamma_0
    =
    \cL_A(\mho)
    =
    P,
    \qquad
    \beta_1 = P^{-1}\mho P^{-1}.
\end{equation}
Here, use is made of an auxiliary matrix
\begin{equation}
\label{CC}
  \Omega
  :=
  C^\rT C,
\end{equation}
so that
$
    \Omega^{\delta_{j1}}
    =
    \left\{
    {\small\begin{matrix}
      \Omega & {\rm if}\ j=1\\
      I_n & {\rm otherwise}
    \end{matrix}}
    \right.
$.
It is convenient to extend (\ref{alf}) to $j=0$ as $\alpha_1 = \gamma_0 \beta_0$ by letting
\begin{equation}
\label{bet0}
    \beta_0:= \gamma_0^{-1} \alpha_1 = I_n,
\end{equation}
 in accordance with the first equality from (\ref{alfbetgam0}). In order for the recurrence equations (\ref{alf}), (\ref{bet}) to be valid for all $j$, it is assumed that the matrices $\gamma_j$ in (\ref{gam}), (\ref{alfbetgam0}) are nonsingular:
 \begin{equation}
 \label{gamdet}
   \det \gamma_j \ne 0,
   \qquad
   j\> 0.
 \end{equation}
In particular, letting $j:=1$ in (\ref{gam}) yields
\begin{equation}
\label{gam1}
    \gamma_1 = \cL_A(\alpha_1 \Omega\gamma_0) = \cL_A(P\Omega P)
\end{equation}
in view of (\ref{alfbetgam0}), and,  since $P\Omega P = PC^\rT (PC^\rT)^\rT$ due to (\ref{CC}), the condition $\det \gamma_1\ne 0$ for the matrix (\ref{gam1}) is equivalent to the controllability of $(A,PC^\rT)$.
The following lemma provides relevant properties of these matrices which will be used in the  proof of Theorem~\ref{th:EOEk}.

\begin{lemma}
\label{lem:alfbetgam}
The matrices $\beta_j$, $\gamma_j$, defined by (\ref{alf})--(\ref{bet0}) subject to (\ref{gamdet}),   are symmetric,  positive semi-definite and positive definite, respectively:
\begin{equation}
\label{betgam}
    \beta_j = \beta_j^\rT\succcurlyeq 0,
    \qquad
    \gamma_j = \gamma_j^\rT \succ 0,
    \qquad
    j \> 0.
\end{equation}
\hfill$\square$
\end{lemma}
\begin{proof}
The symmetry and positive semi-definiteness of the matrices $\beta_0$, $\beta_1$, $\gamma_0$, $\gamma_1$ follow from (\ref{inc2}), (\ref{alfbetgam0}), (\ref{bet0}), (\ref{gam1}) since the matrices $\mho$, $\Omega$ in (\ref{BB}), (\ref{CC}) are symmetric and positive semi-definite, with   $\gamma_0\succ 0$ and $\gamma_1\succ 0$   due to  (\ref{gamdet}).     Therefore, (\ref{bet}) with $j:=1$ leads to
\begin{equation}
\label{bet2}
    \beta_2
    =
    \gamma_1^{-1} \alpha_1 \Omega\gamma_0 \gamma_1^{-1}
    =
    \gamma_1^{-1} P \Omega P \gamma_1^{-1} = \beta_2^\rT \succcurlyeq 0.
\end{equation}
Since $\Omega^{\delta_{j1}} = I_n$ for any $j \> 2$, then,
in view of (\ref{alf}), the recurrence relations (\ref{bet}), (\ref{gam}) reduce to
\begin{align}
\label{betnext}
    \beta_{j+1}
    & =
    \gamma_j^{-1} \overbrace{\gamma_{j-1} \beta_{j-1}}^{\alpha_j} \gamma_{j-1}\gamma_j^{-1},\\
\label{gamnext}
    \gamma_j
    & =
    \cL_A(\gamma_{j-1} \beta_{j-1}  \gamma_{j-1}),
    \qquad
    j\>2,
\end{align}
which, in combination with (\ref{gamdet}) and the inclusions (\ref{incs}), (\ref{inc2}),   implies the properties (\ref{betgam}) by induction.
\end{proof}

Lemma~\ref{lem:alfbetgam} allows the computation of the matrices $\beta_j$ to be organised using the factorisation
\begin{equation}
\label{root}
    \beta_j = \rho_j \rho_j^\rT,
\end{equation}
where, in view of (\ref{betnext}), the appropriately dimensioned (and this time not necessarily symmetric) real matrix square roots   $\rho_j$  satisfy
\begin{equation*}
\label{nextroot}
    \rho_{j+1} = \gamma_j^{-1} \gamma_{j-1} \rho_{j-1},
    \qquad
    j \> 2,
\end{equation*}
with the initial conditions
\begin{equation*}
\label{root12}
    \rho_1:= P^{-1}B,
    \qquad
    \rho_2:= \gamma_1^{-1}PC^\rT,
\end{equation*}
in accordance with (\ref{alfbetgam0}), (\ref{bet2}) and the structure of the matrices $\mho$, $\Omega$ in (\ref{BB}), (\ref{CC}). It follows from (\ref{root}) that the solution $\gamma_j$ of the appropriate ALE in (\ref{gamnext}), with $j\> 2$,  is positive definite if and only if the pair $(A,\gamma_{j-1}\rho_{j-1})$ is controllable. Therefore, the condition (\ref{gamdet}) is equivalent to
\begin{equation}
\label{good}
  (A,B),\
  (A,PC^\rT),\
  (A,\gamma_j\rho_j)\
  {\rm are\ controllable},
  \quad
  j \in \mN.
\end{equation}

\begin{theorem}
\label{th:EOEk}
The powers of the spectral density $S$ in (\ref{Sreal}), with the Hurwitz matrix $A$,    are related to the system $E$  in (\ref{Ereal}) and its dual $E^\sim$ in (\ref{Esimreal}) as
\begin{equation}
\label{EOEk}
    S^k
    =
    C
    \rprod_{j=1}^k
    (\alpha_j^\rT E^\sim)
    \beta_k
    \lprod_{j=1}^k
    (E \alpha_j)
    C^\rT,
    \qquad
    k \in \mN,
\end{equation}
where  $\rprod(\cdot)$, $\lprod(\cdot)$ are the rightwards and leftwards  ordered products, respectively. Here,
the matrices $\alpha_j, \beta_j \in \mR^{n\x n}$ are defined by (\ref{alf})--(\ref{bet0}) subject to the condition (\ref{good}). \hfill$\square$
\end{theorem}
\begin{proof}
Similar to the proof of \cite[Theorem 1]{VP_2022}, a nested induction will be used over $k\in \mN$ (which numbers the outer layer of induction) and $j= 1, \ldots, k-1$ (which numbers the inner layer, becoming inactive at $k=1$).
The validity of (\ref{EOEk}) at $k=1$,  that is,
\begin{equation}
\label{EOE1}
  S
  =
  C \alpha_1^\rT E^\sim \beta_1 E \alpha_1 C^\rT
  =
  C \gamma_0 E^\sim \beta_1 E \alpha_1 C^\rT
\end{equation}
with the matrices $\alpha_1$, $\beta_1$ given by (\ref{alfbetgam0}), follows from (\ref{SEE}) in view of the symmetry and positive definiteness of the covariance matrix $P$ in (\ref{P}) due to the controllability of $(A,B)$. In order to demonstrate  the structure of the subsequent induction steps, consider the left-hand side of (\ref{EOEk}) at $k=2$:
\begin{align}
\nonumber
  S^2
  & =
  SS\\
\nonumber
  & =
  C\alpha_1^\rT E^\sim \beta_1 E \alpha_1
  \Omega
  \gamma_0 E^\sim\beta_1 E\alpha_1C^\rT\\
\label{EOE2}
  & =
  C\alpha_1^\rT
  E^\sim
  \underbrace{\beta_1 \gamma_1}_{\alpha_2^\rT}
  E^\sim
  \underbrace{\gamma_1^{-1} \alpha_1 \Omega\gamma_0 \gamma_1^{-1}}_{\beta_2}
  E
  \underbrace{\gamma_1\beta_1}_{\alpha_2}
  E
  \alpha_1 C^\rT,
\end{align}
where use is made of (\ref{alf})--(\ref{gam}) for $j=1$ along with (\ref{alfbetgam0}), (\ref{CC}), (\ref{gam1}),   (\ref{EOE1}).
The last two equalities in (\ref{EOE2}) involve repeated application of (\ref{EOE}) from Lemma~\ref{lem:EOE}. This allows the rightmost $E^\sim$ factor to be ``pulled'' through the product of the $E$ factors (and constant matrices between them) until the $E^\sim$ factor is to the left of all the $E$ factors. The last equality in (\ref{EOE2}) also uses the symmetry of the matrices $\beta_j$, $\gamma_j$ in (\ref{betgam}) of
Lemma~\ref{lem:alfbetgam}, whereby
\begin{equation}
\label{alfT}
    \alpha_{j+1}^\rT
    =
    (\gamma_j\beta_j)^\rT
    =
    \beta_j\gamma_j,
    \qquad
    j \in \mN.
\end{equation}
Therefore, (\ref{EOE2}) establishes (\ref{EOEk}) for $k=2$. Now, suppose the representation (\ref{EOEk}) is already proved for some $k \> 2$. Then the next power of the matrix $S$ takes the form
\begin{equation}
\label{EOEnext}
    S^{k+1}
    =
    S^k S
    =
    C
    \rprod_{j=1}^k
    (\alpha_j^\rT E^\sim )
    \beta_k
    \lprod_{j=1}^k
    (E \alpha_j)
    \Omega E\mho E^\sim C^\rT ,
\end{equation}
where the rightmost $E^\sim$ factor is the only $E^\sim$ factor which is to the right of the $E$ factors. We will now use the pulling procedure, demonstrated in (\ref{EOE2}),  and prove that
\begin{equation}
\label{inner}
    \beta_k
    \lprod_{j=1}^k
    (E \alpha_j)
    \Omega
    E\mho E^\sim
    =
    \beta_k
    \lprod_{j=r}^k
    (E \alpha_j)
    \gamma_{r-1}
    E^\sim
    \beta_r
    \lprod_{j=1}^r
    (E \alpha_j)
\end{equation}
by induction over $r= 2, \ldots, k$.  The fulfillment of (\ref{inner}) for $r=2$ is verified by applying (\ref{EmhoE}) and using (\ref{alfbetgam0}), (\ref{gam1}):
\begin{align*}
    \beta_k
    \lprod_{j=1}^k
    (E \alpha_j)
    \Omega
    E\mho E^\sim
    & =
    \beta_k
    \lprod_{j=1}^k
    (E \alpha_j)
    \Omega
    P
    E^\sim
    P^{-1} \mho P^{-1}
    E
    P\\
    & =
    \beta_k
    \lprod_{j=1}^k
    (E \alpha_j)
    \Omega
    \gamma_0
    E^\sim
    \beta_1
    E
    \alpha_1\\
    & =
    \beta_k
    \lprod_{j=2}^k
    (E \alpha_j)
    E\alpha_1
    \Omega
    \gamma_0
    E^\sim
    \beta_1
    E
    \alpha_1\\
    & =
    \beta_k
    \lprod_{j=2}^k
    (E \alpha_j)
    \gamma_1
    E^\sim
    \underbrace{\gamma_1^{-1}
    \alpha_1
    \Omega
    \gamma_0
    \gamma_1^{-1}}_{\beta_2}
    E
    \underbrace{\gamma_1
    \beta_1}_{\alpha_2}
    E
    \alpha_1.
\end{align*}
Now, suppose (\ref{inner}) is already proved for some $r= 2, \ldots, k-1$. Then its validity for the next value $r+1$ is established by using (\ref{EOE}) of Lemma~\ref{lem:EOE} in combination with (\ref{alf})--(\ref{gam}) as
\begin{align*}
    \beta_k
    \lprod_{j=1}^k
    (E \alpha_j)
    \Omega
    E\mho E^\sim
    & =
    \beta_k
    \lprod_{j=r+1}^k
    (E \alpha_j)
    E \alpha_r
    \gamma_{r-1}
    E^\sim
    \beta_r
    \lprod_{j=1}^r
    (E \alpha_j)\\
    & =
    \beta_k
    \lprod_{j=r+1}^k
    (E \alpha_j)
    \gamma_r
    E^\sim
    \underbrace{\gamma_r^{-1}\alpha_r
    \gamma_{r-1}
    \gamma_r^{-1}}_{\beta_{r+1}}
    E
    \underbrace{\gamma_r
    \beta_r}_{\alpha_{r+1}}
    \lprod_{j=1}^r
    (E \alpha_j)\\
    & =
    \beta_k
    \lprod_{j=r+1}^k
    (E \alpha_j)
    \gamma_r
    E^\sim
    \beta_{r+1}
    \lprod_{j=1}^{r+1}
    (E \alpha_j).
\end{align*}
Therefore, (\ref{inner}) holds for any $r = 2, \ldots, k$, thus completing the inner layer of induction. In particular, at $r=k$, this relation takes the form
\begin{align}
\nonumber
    \beta_k
    \lprod_{j=1}^k
    (E \alpha_j)
    \Omega
    E\mho E^\sim
    & =
    \beta_k
    E
    \alpha_k\gamma_{k-1}
    E^\sim
    \beta_k
    \lprod_{j=1}^k
    (E \alpha_j)\\
\nonumber
    & =
    \underbrace{\beta_k
    \gamma_k}_{\alpha_{k+1}^\rT}
    E^\sim
    \underbrace{\gamma_k^{-1}
    \alpha_k\gamma_{k-1}
    \gamma_k^{-1}}_{\beta_{k+1}}
    E
    \underbrace{\gamma_k
    \beta_k}_{\alpha_{k+1}}
    \lprod_{j=1}^k
    (E \alpha_j)\\
\label{innerlast}
    & =
    \alpha_{k+1}^\rT
    E^\sim
    \beta_{k+1}
    \lprod_{j=1}^{k+1}
    (E \alpha_j),
\end{align}
where (\ref{EOE}) of Lemma~\ref{lem:EOE} and (\ref{alf})--(\ref{gam}) are used again along  with (\ref{alfT}). Now, substitution of (\ref{innerlast}) into (\ref{EOEnext}) leads to
\begin{align*}
    S^{k+1}
    & =
    C\rprod_{j=1}^k
    (\alpha_j^\rT E^\sim )
    \alpha_{k+1}^\rT
    E^\sim
    \beta_{k+1}
    \lprod_{j=1}^{k+1}
    (E \alpha_j)
    C^\rT\\
    & =
    C\rprod_{j=1}^{k+1}
    (\alpha_j^\rT E^\sim )
    \beta_{k+1}
    \lprod_{j=1}^{k+1}
    (E \alpha_j) C^\rT,
\end{align*}
which completes the outer layer of induction, thus proving (\ref{EOEk}) for any $k\in \mN$.
\end{proof}

For any $\sigma_1, \sigma_2, \sigma_3, \ldots \in \mR\setminus\{0\}$,   the factorisation (\ref{EOEk}) is invariant under the transformation
\begin{equation*}
\label{scale}
    \alpha_k \mapsto \frac{1}{\sigma_k} \alpha_k,
    \quad
    \rho_k \mapsto \rho_k \prod_{j=1}^k \sigma_j,
    \qquad
    k \in \mN,
\end{equation*}
which involves the square root representation of the matrices $\beta_k$ in (\ref{root}) and can be used for balancing the state-space realisations (to improve their numerical computation). 

\section{Hardy-Schatten norms and covariance-analytic functionals in state space}
\label{sec:HS}

In view of the system conjugacy
\begin{equation*}
\label{EEEdual}
        \rprod_{j=1}^k
    (\alpha_j^\rT E^\sim)
    =
    \Big(\lprod_{j=1}^k
    (E \alpha_j)\Big)^\sim,
\end{equation*}
the relation (\ref{EOEk}) can be represented as
\begin{equation}
\label{SG}
  S^k = G_k^\sim G_k,
  \qquad
  k \in \mN,
\end{equation}
in terms of the strictly proper stable systems
\begin{equation}
\label{Gk}
    G_k :=
    \rho_k^\rT
    \Big(
        \lprod_{j=1}^k
        E \alpha_j
    \Big)
    C^\rT,
    \qquad
    k \in \mN,
\end{equation}
which are assembled into
\begin{equation}
\label{G}
    G
    :=
    \begin{bmatrix}
      G_1\\
      G_2\\
      G_3\\
      \vdots
    \end{bmatrix}
    =
    \begin{bmatrix}
      \rho_1^\rT \\
      \rho_2^\rT E\alpha_2 \\
      \rho_3^\rT E\alpha_3 E\alpha_2 \\
      \vdots
    \end{bmatrix}
    E\alpha_1 C^\rT
    =
    \left[\begin{array}{cccc|c}
  A   & 0 & 0 &\ldots & \alpha_1 C^\rT\\
  \alpha_2  & A & 0 & \ldots  & 0\\
  0   & \alpha_3 & A & \ldots  & 0\\
  \ldots    &\ldots   & \ldots  & \ldots  &\ldots  \\
  \hline
  \rho_1^\rT & 0 & 0 & \ldots  & 0\\
  0 & \rho_2^\rT & 0 & \ldots  & 0\\
  0 & 0 & \rho_3^\rT & \ldots  & 0\\
  \ldots & \ldots & \ldots & \ldots  & \ldots
\end{array}\right],
\end{equation}
provided the condition (\ref{good}) is satisfied.
The system $G$ has infinite-dimensional real state-space matrices  (its output matrix is an infinite block-diagonal matrix involving the square roots from (\ref{root})),
an $\mR^p$-valued input and an $\mR^\infty $-valued output. This system is organised as an infinite cascade of systems shown in Fig.~\ref{fig:EEE}.
\begin{figure}[htbp]
\centering
\unitlength=0.75mm
\linethickness{0.5pt}
\begin{picture}(140.00,44.00)
    \put(0,7){\dashbox(135,35)[cc]{}}
    \put(3,40){\makebox(0,0)[lt]{{$G$}}}
    \put(120,27){\framebox(10,10)[cc]{$\alpha_1 C^\rT$}}
    \put(100,27){\framebox(10,10)[cc]{$E$}}
    \put(80,27){\framebox(10,10)[cc]{$\alpha_2$}}
    \put(60,27){\framebox(10,10)[cc]{$E$}}
    \put(40,27){\framebox(10,10)[cc]{$\alpha_3$}}
    \put(20,27){\framebox(10,10)[cc]{$E$}}
    \put(142,32){\vector(-1,0){12}}
    \put(120,32){\vector(-1,0){10}}
    \put(100,32){\vector(-1,0){10}}
    \put(80,32){\vector(-1,0){10}}
    \put(60,32){\vector(-1,0){10}}
    \put(40,32){\vector(-1,0){10}}
    \put(20,32){\vector(-1,0){10}}
    \put(5,32){\makebox(0,0)[cc]{{$\cdots$}}}
    \put(5,2){\makebox(0,0)[cc]{{$\cdots$}}}

    \put(95,32){\vector(0,-1){10}}
    \put(90,12){\framebox(10,10)[cc]{$\rho_1^\rT$}}
    \put(95,12){\vector(0,-1){12}}
    \put(55,32){\vector(0,-1){10}}
    \put(50,12){\framebox(10,10)[cc]{$\rho_2^\rT$}}
    \put(55,12){\vector(0,-1){12}}

    \put(15,32){\vector(0,-1){10}}
    \put(10,12){\framebox(10,10)[cc]{$\rho_3^\rT$}}
    \put(15,12){\vector(0,-1){12}}

\end{picture}
\caption{An infinite cascade of copies of the system $E$ from (\ref{Ereal}) (with the matrices $\alpha_k$ from (\ref{alf}), (\ref{alfbetgam0}) as intermediate factors)  forming the system $G$ with the state-space realization (\ref{G}).
}
\label{fig:EEE}
\end{figure}
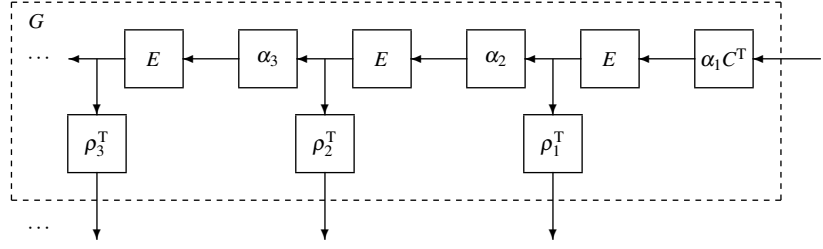
The significance of the systems $G_k$ in (\ref{Gk}) for computing the norms $\|F\|_{2k}$ is clarified by the following theorem. Its formulation employs matrices $P_{jk} \in \mR^{n\x n}$ satisfying a recurrence equation
\begin{equation}
\label{Pnext}
  P_{jk}
   =
    \left\{
    \begin{array}{lll}
    \cL_A(P_{1, k-1}\alpha_k^\rT) & {\rm if}&  j=1\\
   \cL_A(\alpha_j P_{j-1, k} +  P_{j, k-1}\alpha_k^\rT)     & {\rm if}&  1 < j< k\\
    \cL_A(\alpha_k P_{k-1, k} +  P_{k-1, k}^\rT\alpha_k^\rT) & {\rm if}&  j=k
    \end{array}
    \right.,
    \quad
    1\< j\< k,
    \quad
    k > 1,
\end{equation}
which uses the operator (\ref{cL}) and the matrices (\ref{alf}), (\ref{CC}) along with the initial condition
\begin{equation}
\label{P11}
  P_{11} := \gamma_1
\end{equation}
given by (\ref{gam1}).

\begin{theorem}
\label{th:FG}
Suppose the linear stochastic system, described by (\ref{XZ})--(\ref{Es}), with $A$ Hurwitz, satisfies the condition (\ref{good}). Then its Hardy-Schatten norms (\ref{H2k}) are related to the $\cH_2$-norms of the systems (\ref{Gk}) as
\begin{equation}
\label{H2kG}
  \|F\|_{2k}
  =
  \sqrt[k]{\|G_k\|_2}
  =
  \sqrt[2k]{
      \bra
        \beta_k, P_{kk}
    \ket_\rF},
  \qquad
  k \in \mN,
\end{equation}
with the matrices $\beta_k$, $P_{kk}$ from (\ref{bet}), (\ref{Pnext}), (\ref{P11}).
\hfill$\square$
\end{theorem}
\begin{proof}
Integration of the transfer functions on both sides of (\ref{SG}) over the imaginary axis and a comparison with (\ref{H2k}), (\ref{H2}) (with the $\cH_2$-norm being evaluated at $G_k$) yield
\begin{align}
\nonumber
    \|F\|_{2k}^{2k}
     & =
    \frac{1}{2\pi}
    \int_{\mR}
    \Tr(S(\omega)^k)
    \rd \omega
    =
    \frac{1}{2\pi}
    \int_{\mR}
    \Tr(G_k(i\omega)^*G_k(i\omega))
    \rd \omega\\
\label{FGk}
    & =
    \frac{1}{2\pi}
    \int_{\mR}
    \|G_k(i\omega)\|_\rF^2
    \rd \omega = \|G_k\|_2^2
\end{align}
for any     $k \in \mN$,
which establishes the first equality in (\ref{H2kG}). Now, the $\cH_2$-norm of the system $G_k$ can be computed as
\begin{equation}
\label{Gk2}
    \|G_k\|_2^2
    =
    \Tr (\rho_k^\rT P_{kk} \rho_k)
    =
    \bra
        \beta_k, P_{kk}
    \ket_\rF
\end{equation}
in terms of the matrices $\beta_k$, $\rho_k$ from (\ref{bet}), (\ref{root}) and the blocks
$P_{jk}  = P_{kj}^\rT \in \mR^{n\x n}$ of the controllability Gramian
\begin{equation}
\label{cPN}
    \cP_N
    :=
    (P_{jk})_{1\< j,k\< N}
    =
    \cL_{\cA_N}(\cB_N\cB_N^\rT)
\end{equation}
of $(\cA_N, \cB_N)$. Here,
$\cA_N \in \mR^{nN\x nN}$, $\cB_N\in \mR^{nN\x p}$, $\cC_N\in \mR^{(m + p(N-1))\x nN}$  are the state-space matrices of the system
\begin{equation}
\label{GN}
    \cG_N
    :=
    \begin{bmatrix}
      G_1\\
      \vdots\\
      G_N
    \end{bmatrix}
    =
    \left[
    \begin{array}{c|c}
    \cA_N & \cB_N\\
      \hline
      \cC_N &  0
    \end{array}
    \right],
\end{equation}
obtained by truncating the system $G$ in (\ref{G}), so that
\begin{equation}
\label{cABCN}
    \cA_N
    :=
    \begin{bmatrix}
        A   & 0 & 0 &\ldots &\ldots & 0\\
        \alpha_2  & A & 0 & \ldots  &\ldots & 0\\
        0   & \alpha_3 & A & \ldots  &\ldots & 0\\
        \ldots    &\ldots   & \ldots  & \ldots  &\ldots &  \ldots\\
        \ldots    &\ldots   & \ldots  & \ldots  &A &  0\\
        \ldots    &\ldots   & \ldots  & \ldots  &\alpha_N &  A
    \end{bmatrix},
    \quad
    \cB_N
    :=
    \begin{bmatrix}
      \alpha_1 C^\rT\\
      0\\
      \vdots\\
      0
    \end{bmatrix},
    \quad
    \cC_N := \diag_{1\< k \< N} (\rho_k^\rT),
\end{equation}
with $\cA_N$ inheriting the Hurwitz property from $A$. In (\ref{cPN}), use is made of the property that the matrix  $\cP_N$ is a submatrix of $\cP_{N+1}$ for any $N \in \mN$   due to the cascade structure of the system $G$ (see Fig.~\ref{fig:EEE}).   Also, (\ref{Gk2}) makes advantage of the block-diagonal structure of the matrix $\cC_N$ in (\ref{cABCN}), whereby the $k$th diagonal block of the matrix $\cC_N \cP_N \cC_N^\rT$ reduces to     $(\cC_N \cP_N  \cC_N^\rT)_{kk} = \rho_k^\rT P_{kk} \rho_k$.
 The block lower triangular structure of $\cA_N$ and the sparsity of $\cB_N$ in (\ref{cABCN}) allow the $(j,k)$th block of the ALE in (\ref{cPN}) to be represented as
\begin{align}
\nonumber
    0 &=
    (\cA_N \cP_N + \cP_N \cA_N^\rT  + \cB_N \cB_N^\rT)_{jk}\\
\nonumber
    & =
    (\cA_N \cP_N)_{jk} + ((\cA_N \cP_N)_{kj})^\rT + (\cB_N \cB_N^\rT)_{jk}\\
\label{Pjk}
    & =
    AP_{jk} + P_{jk}A^\rT + \alpha_j P_{j-1, k} +  P_{j, k-1}\alpha_k^\rT
    +
    \delta_{j1} \delta_{k1} P\Omega P,
    \qquad
    1\< j,k\< N,
\end{align}
where the convention $P_{0k}=0$,  $P_{k0}=0$ is used along with (\ref{alfbetgam0}), (\ref{CC}). It follows from (\ref{Pjk}) that
\begin{equation}
\label{Pjk1}
  P_{jk}
   =
   \cL_A(\alpha_j P_{j-1, k} +  P_{j, k-1}\alpha_k^\rT
    +
    \delta_{j1} \delta_{k1} P\Omega P).
\end{equation}
In the case of $j=k=1$, this relation reduces to
$
  P_{11}
   =
   \cL_A(P\Omega P)
$,
and hence, $P_{11}$ coincides with the matrix $\gamma_1$ in (\ref{gam1}), in accordance with (\ref{P11}). For any $k>1$, the recurrence equation (\ref{Pnext}) is obtained by considering (\ref{Pjk1}) for $j=1$, $j = 2, \ldots, k-1$ and $j=k$, which, together with (\ref{Gk2}), completes the proof of the second equality in (\ref{H2kG}).
\end{proof}

In combination with (\ref{FGk}), the representation (\ref{JHS}) of the covariance-analytic functional (\ref{phiFF}) leads to
\begin{equation}
\label{JHSG}
  J_\phi(F)
  =
  \sum_{k=1}^{+\infty}
  \phi_k \|G_k\|_2^2,
\end{equation}
thus extending the state-space calculations
from monomials to analytic functions of  the spectral density $S$. In a particular yet practically important case, where all the derivatives of $\phi$ in (\ref{JHS})  are nonnegative
\begin{equation}
\label{phipos}
    \phi_k \> 0,
     \qquad
     k \in \mN,
\end{equation}
which holds, for example, in the risk-sensitive setting\footnote{and any other control formulation which penalises the Hardy-Schatten norms with positive weights}  (\ref{QEphi}), the functional (\ref{JHSG}) is related by
\begin{equation}
\label{JH}
  J_\phi(F)
  =
  \sum_{k=1}^{+\infty}
  \|H_k\|_2^2
  =
  \|H\|_2^2
\end{equation}
to the $\cH_2$-norms of the systems
\begin{equation*}
\label{Hk}
  H_k:= \sqrt{\phi_k} G_k
\end{equation*}
which comprise the system
\begin{equation}
\label{GH}
    H
    :=
    \begin{bmatrix}
      H_1\\
      H_2\\
      H_3\\
      \vdots
    \end{bmatrix}
    =
    \left[\begin{array}{cccc|c}
  A   & 0 & 0 &\ldots & \alpha_1 C^\rT\\
  \alpha_2  & A & 0 & \ldots  & 0\\
  0   & \alpha_3 & A & \ldots  & 0\\
  \ldots    &\ldots   & \ldots  & \ldots  &\ldots  \\
  \hline
  \sqrt{\phi_1}\rho_1^\rT & 0 & 0 & \ldots  & 0\\
  0 & \sqrt{\phi_2}\rho_2^\rT & 0 & \ldots  & 0\\
  0 & 0 & \sqrt{\phi_3}\rho_3^\rT & \ldots  & 0\\
  \ldots & \ldots & \ldots & \ldots  & \ldots
\end{array}\right],
\end{equation}
whose state-space realisation is obtained from (\ref{G}). The representation (\ref{JH}) of the covariance-analytic cost for the finite-dimensional system $F$ in (\ref{F0}) in terms of the   LQG cost of an auxiliary system $H$, associated with $F$ by (\ref{GH}),  can be regarded as a system theoretic counterpart of the sum-of-squares (SOS) structures from polynomial  optimization \cite{L_2001,P_2000}.

The practical computation of the series (\ref{JHSG}) (or (\ref{JH}) under the condition (\ref{phipos})) can be carried out by truncating it to the first $N$ terms, with $N$ being  moderately large in the case of fast decaying coefficients $\phi_k$. As discussed in the proof of Theorem~\ref{th:FG}, in addition to solving $N$ ALEs of order $n$ for finding the auxiliary matrices $\alpha_1, \ldots, \alpha_N$, $\beta_1, \ldots, \beta_N$ in (\ref{alf})--(\ref{alfbetgam0}), this involves one ALE of order $nN$ (or  $\frac{1}{2}N(N+1)$ ALEs of order $n$ in (\ref{Pnext}), (\ref{P11})) for computing the controllability Gramian $\cP_N$ in (\ref{cPN}) for the system $\cG_N$ in (\ref{GN}).

\section{An alternative  computation with differentiating a Riccati equation}
\label{sec:diff}

For comparison with the results of Section~\ref{sec:HS}, we will now discuss  a different approach to the state-space computation of the covariance-analytic functional and related Hardy-Schatten norms.
To this end, the following lemma rephrases \cite[Proposition 6.3.1, pp. 66--68]{MG_1990} (see also \cite{BV_1985}) in application to the risk-sensitive cost (\ref{QE}) and is given here for completeness.

\begin{lemma}
\label{lem:QEF}
Suppose the  risk sensitivity parameter $\theta>0$ satisfies (\ref{tgood}) with the stable transfer function (\ref{F0}). Then the risk-sensitive cost (\ref{QE}) for the  stationary Gaussian process $Z$ in (\ref{XZ}) can be  computed  as
\begin{equation}
\label{QEFrate}
    \Xi(\theta)
    =
    \frac{1}{2}
    \bra
        \mho,
        \Psi(\theta)
    \ket_\rF.
\end{equation}
Here, $\Psi(\theta)\in \mS_n$ is the unique stabilising solution of the algebraic Riccati equation (ARE)
\begin{equation}
\label{ARE}
    A^\rT \Psi(\theta) + \Psi(\theta)A
   +
   \theta \Omega
   +
    \Psi(\theta) \mho \Psi(\theta) = 0
\end{equation}
in the sense that the matrix
\begin{equation}
\label{Ups}
  \Ups(\theta)
  :=
  A + \mho \Psi(\theta)
\end{equation}
is Hurwitz, with the matrices $\mho$, $\Omega$ given by (\ref{BB}), (\ref{CC}).
\hfill$\blacksquare$
\end{lemma}

Since the matrix $A$ is Hurwitz,   the stabilising solution $\Psi(\theta)$ of the $\theta$-dependent ARE (\ref{ARE}) satisfies
\begin{equation}
\label{R0}
  \Psi(0) = 0.
\end{equation}
Also, $\Psi(\theta)$ is infinitely differentiable in $\theta$ over the interval (\ref{tgood}). In view of (\ref{FXi}), the Hardy-Schatten norms can be recovered from the state-space representation (\ref{QEFrate}) of the risk-sensitive cost as
\begin{equation}
\label{FXi1}
    \|F\|_{2k}^{2k}
    =
    \frac{1}{(k-1)!}
    \bra
        \mho,
        \Psi_k
    \ket_\rF,
    \qquad
    \Psi_k:= \Psi^{(k)}(0),
    \qquad
    k \in \mN.
\end{equation}
By substituting (\ref{FXi1}) into (\ref{JHS}), the
covariance-analytic functional in (\ref{phiFF}) takes the form
\begin{equation}
\label{JR}
    J_\phi(F)
    =
    \Big\bra
        \mho,\,
    \sum_{k=1}^{+\infty}
        \frac{\phi_k}{(k-1)!}
        \Psi_k
    \Big\ket_\rF.
\end{equation}
This alternative way for computing the cost $J_\phi(F)$ and the $\cH_{2k}$-norms   in state space  involves the derivatives $\Psi^{(k)}$ of the stabilising solution of (\ref{ARE}), which can be recursively found as follows.

\begin{theorem}
\label{th:diff}
Under the conditions of Lemma~\ref{lem:QEF}, the derivatives $\Psi_k$  of the stabilising solution $\Psi$ of the ARE (\ref{ARE}) in (\ref{FXi1}) satisfy a bilinear recurrence equation
\begin{equation}
\label{Psik}
  \Psi_k = \cL_{A^\rT}(\Omega_k),
  \qquad
  \Omega_k
  :=
  \delta_{k1}\Omega
  +
      \sum_{j=1}^{k-1}
    \binom{k}{j}
    \Psi_j \mho \Psi_{k-j},
    \qquad
    k \in \mN,
\end{equation}
where the operator (\ref{cL}) is used along with the matrices $\mho$, $\Omega$ from (\ref{BB}), (\ref{CC}) and the binomial coefficients.
\hfill$\square$
\end{theorem}
\begin{proof}
Since the matrices $A$, $\mho$, $\Omega$ are constant, the differentiation of both sides of the ARE (\ref{ARE}) in $\theta$ yields
\begin{equation*}
\label{ARE'}
    0 =
    A^\rT \Psi' + \Psi'A
   +
   \Omega
   +
    \Psi'\mho \Psi + \Psi\mho \Psi'
    =
    \Ups^\rT \Psi' + \Psi' \Ups + \Omega,
\end{equation*}
and hence,
\begin{equation}
\label{R'}
  \Psi' = \cL_{\Ups^\rT}(\Omega),
\end{equation}
where the operator (\ref{cL}) is used along with the Hurwitz matrix $\Ups$ from (\ref{Ups}), and the argument $\theta$ of $\Psi(\theta)$, $\Ups(\theta)$ is omitted for brevity.  Now, extending (\ref{R'}), suppose it is already proved for some $k \in \mN$ that
\begin{equation}
\label{Rk1}
    \Psi^{(k)} = \cL_{\Ups^\rT}(\Theta_k),
    \qquad
    \Theta_k(\theta)
  :=
  \delta_{k1}\Omega
  +
      \sum_{j=1}^{k-1}
    \binom{k}{j}
    \Psi^{(j)} \mho \Psi^{(k-j)}
\end{equation}
for any $\theta$ satisfying (\ref{tgood}) (note that in the case of $k=1$, the sum vanishes:  $\sum_{j=1}^0 = 0$).
Since $\Theta_k$ in (\ref{Rk1}) and $\Ups$ in (\ref{Ups}) inherit smoothness from $\Psi$,  with
$    \Ups' = \mho \Psi'
$,
then by differentiating the ALE
$$
    \Ups^\rT \Psi^{(k)}+\Psi^{(k)}\Ups + \Theta_k = 0
$$
with respect to $\theta$, it follows that
\begin{align*}
    0 & = (\Ups^\rT \Psi^{(k)}+\Psi^{(k)}\Ups + \Theta_k)' \\
    & =
    \Ups^\rT \Psi^{(k+1)}+\Psi^{(k+1)}\Ups + \Psi^{(k)}\mho \Psi' + \Psi' \mho \Psi^{(k)} + \Theta_k',
\end{align*}
so that
\begin{equation}
\label{Rnext}
    \Psi^{(k+1)} = \cL_{\Ups^\rT}(\Theta_{k+1}),
    \qquad
    \Theta_{k+1} := \Psi^{(k)}\mho \Psi' + \Psi' \mho \Psi^{(k)} + \Theta_k'.
\end{equation}
Substitution of the matrix $\Theta_k$ from (\ref{Rk1}) into the second equality in (\ref{Rnext}) leads to
\begin{align}
  \nonumber
  & \Theta_{k+1}
  =
   \Psi^{(k)}\mho \Psi' + \Psi' \mho \Psi^{(k)}
   +
  \sum_{j=1}^{k-1}
    \binom{k}{j}
    (\Psi^{(j+1)} \mho \Psi^{(k-j)}+ \Psi^{(j)} \mho \Psi^{(k-j+1)})\\
  \nonumber
    & =
   \Psi^{(k)}\mho \Psi' + \Psi' \mho \Psi^{(k)}
   +
     \sum_{j=2}^{k}
    \binom{k}{j-1}
    \Psi^{(j)} \mho \Psi^{(k+1-j)}
    +
  \sum_{j=1}^{k-1}
    \binom{k}{j}
    \Psi^{(j)} \mho \Psi^{(k+1-j)}        \\
  \nonumber
    & =
  \sum_{j=1}^{k}
  \Big(
      \binom{k}{j-1}
      +
    \binom{k}{j}
  \Big)
    \Psi^{(j)} \mho \Psi^{(k+1-j)}    \\
\label{Tnext}
    & =
  \sum_{j=1}^{k}
    \binom{k+1}{j}
    \Psi^{(j)} \mho \Psi^{(k+1-j)},
\end{align}
where the last equality uses the Pascal triangle identity for the binomial coefficients. Together with the first equality in (\ref{Rnext}), the relation (\ref{Tnext}) provides the induction step which proves (\ref{Rk1}) for any $k \in \mN$. The recurrence equation (\ref{Psik}) can now be obtained by considering (\ref{Rk1}) at $\theta= 0$  and using the property
$$
    \Ups(0) = A + \mho \Psi(0) = A,
$$
which follows from (\ref{Ups}), (\ref{R0}) and leads to $\Omega_k  = \Theta_k(0)$. It now remains to note that, according to (\ref{Psik}),     for any $k>1$, the matrix $\Psi_k$ depends on $\Psi_1, \ldots, \Psi_{k-1}$ in a bilinear fashion.
\end{proof}

Since the matrix  $\Omega_1$ in (\ref{Psik}) with $k=1$ reduces to $\Omega$ in (\ref{CC}), the matrix $\Psi_1$ takes the form
\begin{equation*}
\label{R1}
  \Psi_1 = \cL_{A^\rT}(\Omega),
\end{equation*}
and,  in view of (\ref{cL}),  coincides with the observability Gramian $Q$ of the pair $(A,C)$ satisfying the ALE
\begin{equation*}
\label{QALE}
    A^\rT Q + Q A + \Omega = 0,
\end{equation*}
which is dual to the ALE (\ref{PALE}). Accordingly, the dual form of the relations (\ref{QEFrate}), (\ref{ARE}) of Lemma~\ref{lem:QEF} and (\ref{Psik}) of Theorem~\ref{th:diff} is obtained by replacing the matrix triple $(A,B,C)$ with $(A^\rT, C^\rT, B^\rT)$  and is given by
\begin{equation*}
\label{QEFrate1}
    \Xi(\theta)
    =
    \frac{1}{2}
    \bra
        \Omega,
        \Phi(\theta)
    \ket_\rF,
\end{equation*}
where $\Phi(\theta)\in \mS_n$ is the unique stabilising (with $A + \Phi(\theta) \Omega$ Hurwitz) solution of the ARE
\begin{equation*}
\label{ARE1}
    A \Phi(\theta) + \Phi(\theta)A^\rT
   +
   \theta \mho
   +
    \Phi(\theta) \Omega \Phi(\theta) = 0,
\end{equation*}
whose derivatives satisfy the recurrence equation
\begin{equation}
\label{Phik}
  \Phi_k := \Phi^{(k)}(0) =  \cL_{A}(\mho_k),
  \qquad
  \mho_k
  :=
  \delta_{k1}\mho
  +
      \sum_{j=1}^{k-1}
    \binom{k}{j}
    \Phi_j \Omega \Phi_{k-j},
    \qquad
    k \in \mN,
\end{equation}
and give rise to the corresponding dual representations for the Hardy-Schatten norms in (\ref{FXi1}),
\begin{equation}
\label{FXi2}
    \|F\|_{2k}^{2k}
    =
    \frac{1}{(k-1)!}
    \bra
        \Omega,
        \Phi_k
    \ket_\rF,
    \qquad
    k \in \mN,
\end{equation}
and the covariance analytic functional in (\ref{JR}):
\begin{equation*}
\label{JR1}
    J_\phi(F)
    =
    \Big\bra
        \Omega,\,
    \sum_{k=1}^{+\infty}
        \frac{\phi_k}{(k-1)!}
        \Phi_k
    \Big\ket_\rF.
\end{equation*}
By a similar reasoning, the matrix  $\mho_1$ in (\ref{Phik}) at $k=1$ reduces to $\mho$ in (\ref{BB}), so that the matrix $\Phi_1$ takes the form
\begin{equation*}
\label{Phi1}
  \Phi_1 = \cL_A(\mho),
\end{equation*}
and, in view of (\ref{P}),  coincides with the controllability Gramian $P$ of the pair $(A,B)$. Therefore, the role of the matrices $\Phi_k$, $\Psi_k$ in computing the higher-order $\cH_{2k}$-norm  $\|F\|_{2k}$ (with $k>1$) is similar to that of the Gramians $P= \Phi_1$, $Q= \Psi_1$ for the mean square cost
$$
    \|F\|_2^2
    =
    \bra \Omega, P\ket_\rF
    =
    \bra \mho, Q\ket_\rF.
$$
For example, in accordance with (\ref{FXi1}), (\ref{FXi2}), the $\cH_4$-norm is related to $\Phi_2$, $\Psi_2$ as
$$
    \|F\|_4^4
    =
    \bra \Omega, \Phi_2\ket_\rF
    =
    \bra \mho, \Psi_2\ket_\rF,
$$
where
$$
    \Phi_2 = 2\cL_A(P\Omega P),
    \qquad
    \Psi_2 = 2\cL_{A^\rT}(Q\mho Q)
$$
(cf. \cite[p. 2386]{VP_2010b}),  so that $\frac{1}{2}\Phi_2$, $\frac{1}{2}\Psi_2$ are the controllability and observability Gramians for a subsidiary system
$$
      \left[
    \begin{array}{c|c}
    A & PC^\rT\\
      \hline
      B^{\rT}Q &  0
    \end{array}
    \right],
$$
with $\frac{1}{2}\Phi_2 = \gamma_1$ in view of (\ref{gam1}).
By analogy with the Gramians $P$, $Q$, the matrices $\Phi_k$, $\Psi_k$ are referred to as the controllability and observability Schattenians \cite{VP_2010b} of order $k$ for the triple $(A,B,C)$, respectively.

In combination with (\ref{FXi1}),  Theorem~\ref{th:diff} (or the dual version of (\ref{Psik}) in (\ref{Phik})) allows the first $N$ Hardy-Schatten norms $\|F\|_2, \|F\|_4, \ldots, \|F\|_{2N}$  to be computed by solving $N$ ALEs of order $n$, which is simpler than the approach of Section~\ref{sec:HS} based on the adaptation  of the system transposition technique in  Section~\ref{sec:inf}. However, the latter approach contributes a special class of spectral factorizations, which overcomes the noncommutativity  of transfer matrices  in cascaded linear systems and is of theoretical interest in its own right.

\section{A numerical example of computing the Hardy-Schatten norms}
\label{sec:num}

As an illustration, consider a system (\ref{Ereal})  with input, state and output dimensions $m = 2$, $n = 4$, $p = 3$ and the state-space realization
$$
    F
    =
      \left[
    \begin{array}{c|c}
    A & B\\
      \hline
      C &  0
    \end{array}
    \right]
    =
      \left[
    {\small\begin{array}{cccc|cc}
   -1.2018 &  -0.5741 &  -1.1385 &  -0.2364 &   0.1024 &  -2.0210\\
   -0.2060 &  -0.8052 &  -0.4136 &  -0.5226 &   1.7795 &  -1.6236\\
    0.4992 &   0.1745 &   0.0814 &   1.6367 &  -0.3501 &  -1.4740\\
    0.9253 &   0.6402 &  -0.4446 &  -1.8185 &   0.2054 &   1.5669\\
    \hline
    0.7726 &   1.9030 &   0.8135 &   0.3452 &        0 &        0\\
   -0.5229 &  -0.8653 &  -1.2256 &   0.6073 &        0 &        0\\
    0.2605 &   0.4777 &  -1.2256 &  -1.4299 &        0 &        0
    \end{array}}
    \right],
$$
whose matrices $A$, $B$, $C$ are generated randomly, with $A$ Hurwitz (its
spectrum is $\{-0.5409 \pm 1.2631i,   -1.9875,   -0.6748\}$).  Such a model can represent a  mechanical system with two degrees of freedom, resulting from an internally stable interconnection of a plant and a controller, each organised as a mass-spring-damper system  with one degree of freedom and subject to an external random force. The results of the state-space computation of the first ten Hardy-Schatten norms for this system by using Theorem~\ref{th:FG} are shown in Fig.~\ref{fig:h2k}.
\begin{figure}[htbp]
\begin{center}
\includegraphics[width=8cm]{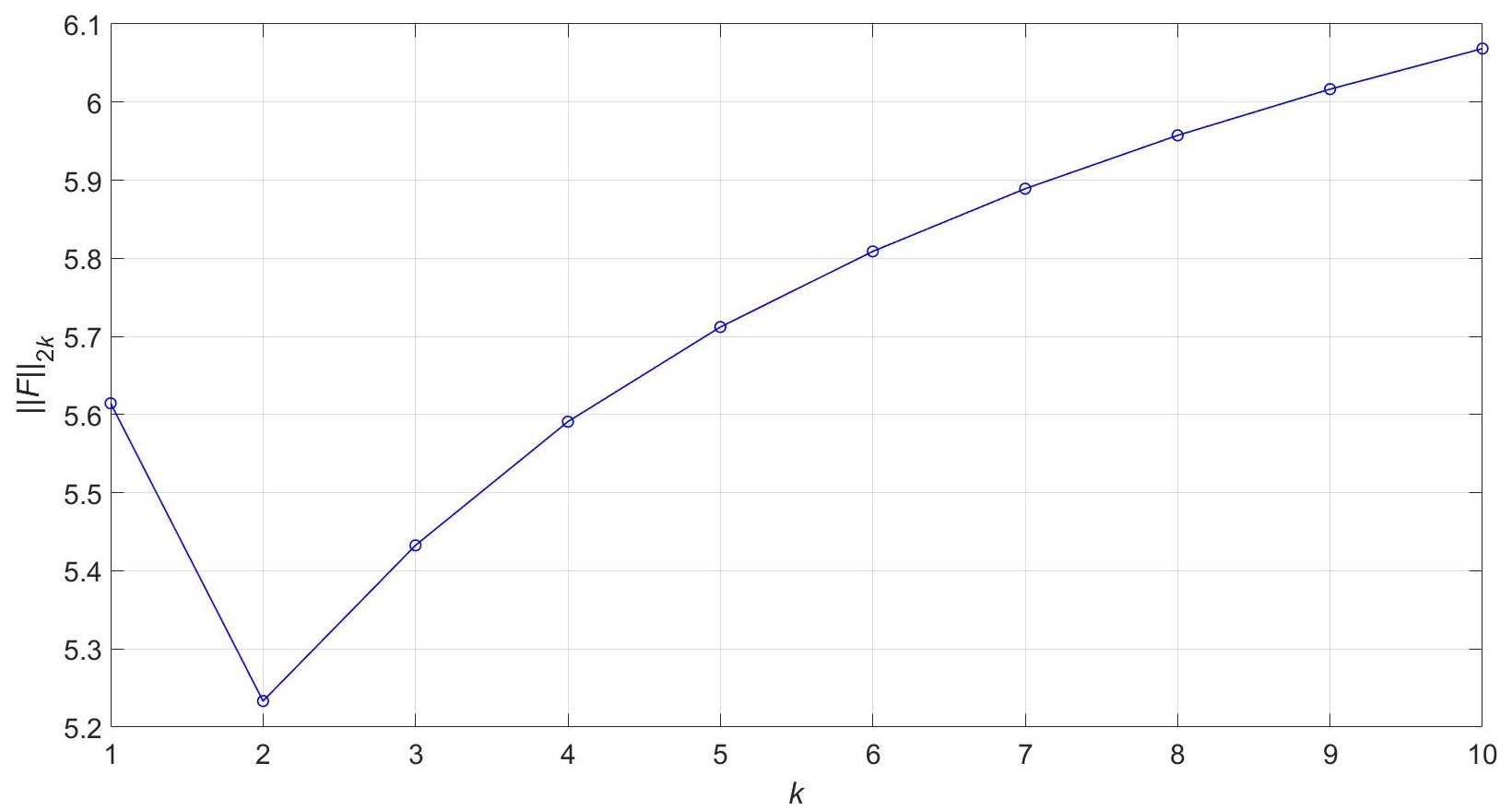}
\caption{The first ten Hardy-Schatten norms $\|F\|_{2k}$, with $k= 1, \ldots, 10$, for the model system computed using Theorem~\ref{th:FG}.  }
\label{fig:h2k}
\end{center}
\end{figure}
In Fig.~\ref{fig:hhgap}, 
these results are compared (in terms of a relative gap) with the alternative method of computing the norms through the combination of (\ref{FXi1}) and Theorem~\ref{th:diff}. This comparison provides an experimental cross-verification of both approaches.

\begin{figure}[htbp]
\begin{center}
\includegraphics[width=8cm]{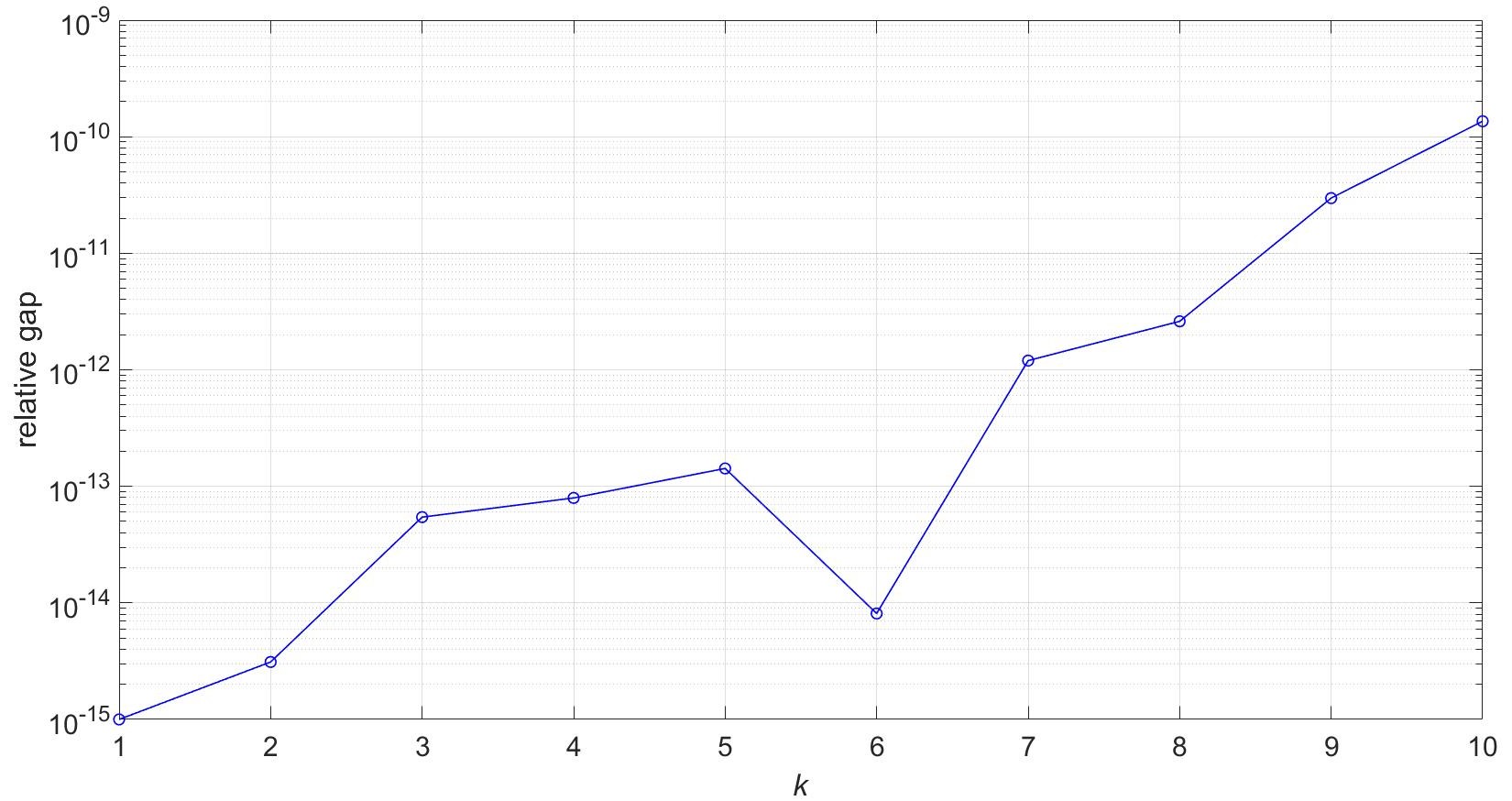}
\caption{The relative gap (on the logarithmic scale) between the norms $\|F\|_{2k}$, with $k= 1, \ldots, 10$, of the model system computed in two ways: using Theorem~\ref{th:FG} and the alternative approach of (\ref{FXi1}) together with  Theorem~\ref{th:diff}.  }
\label{fig:hhgap}
\end{center}
\end{figure}

\section{Conclusion}
\label{sec:conc}

For linear stochastic systems with a stationary Gaussian output process, we have considered a class of covariance-analytic performance criteria, which  involve the trace of a cost-shaping  function of the output spectral density and contain the standard mean square and risk-sensitive costs as particular cases. With such a cost being closely related to the Hardy-Schatten norms  of the system, we have discussed their links with the asymptotic growth rates of the output energy cumulants. We have also considered the worst-case mean square values of the system output in the case of statistical uncertainty about the input noise using covariance-analytic functionals with convex conjugate pairs of cost-shaping functions.  For a class of strictly proper finite-dimensional systems, with state-space dynamics  governed by linear SDEs,  we have developed a method for recursively computing the Hardy-Schatten norms through a modification of a recently proposed technique of rearranging cascaded linear systems, reminiscent of the Wick ordering of quantum mechanical annihilation and creation operators. This computational procedure involves a recurrence sequence of solutions to ALEs and represents the covariance-analytic functional as the squared  $\cH_2$-norm of an auxiliary cascaded system.
These results have been compared with a different approach using higher-order derivatives of stabilising solutions of a parameter-dependent ARE and illustrated by a numerical example. 


\end{document}